\documentclass[10pt]{article} 
\usepackage{hyperref}
\usepackage{url}

\usepackage{algorithmic, algorithm}
\usepackage{amsfonts}
\usepackage{amsmath}
\usepackage{amsthm}
\usepackage{color}
\usepackage{fullpage}
\usepackage{graphicx}

\newtheorem{theorem}{Theorem}[section]
\newtheorem{proposition}[theorem]{Proposition}
\newtheorem{definition}[theorem]{Definition}
\newtheorem{lemma}[theorem]{Lemma}
\newtheorem{corollary}[theorem]{Corollary}

\providecommand{\st}{\ensuremath{\textrm{s.t.}\quad}}

\providecommand{\ones}{\mathbf{1}}

\providecommand{\bR}{\mathbb R}
\providecommand{\cB}{\mathcal{B}}
\providecommand{\cP}{\mathcal{P}}
\providecommand{\cPH}{\mathcal{PH}}
\providecommand{\cSN}{\mathcal{SN}}

\newcommand{\Tr}{\ensuremath{\rm{Trace}}}

\renewcommand{\st}{\text{such that}}

\title{Sorting Network Relaxations for Vector Permutation Problems}
\author{
Cong Han Lim, Stephen J. Wright \\
Computer Sciences Department, University of Wisconsin - Madison \\
\texttt{conghan@cs.wisc.edu, swright@cs.wisc.edu} \\
}

\definecolor{antiquefuchsia}{rgb}{0.57, 0.36, 0.51}
\definecolor{darkgreen}{rgb}{0.0, 0.5, 0.0}

\begin{document}

\maketitle

\begin{abstract}
The Birkhoff polytope (the convex hull of the set of permutation
matrices) is frequently invoked in formulating relaxations of
optimization problems over permutations. The Birkhoff polytope is
represented using $\Theta(n^2)$ variables and constraints,
significantly more than the $n$ variables one could use to represent a
permutation as a vector. Using a recent construction of Goemans
\cite{Goemans2010}, we show that when optimizing over the convex hull
of the permutation vectors (the permutahedron), we can reduce the
number of variables and constraints to $\Theta(n \log n)$ in theory
and $\Theta(n \log^2 n)$ in practice. We modify the recent convex
formulation of the 2-SUM problem introduced by Fogel
et~al.~\cite{Fogel2013} to use this polytope, and demonstrate how we
can attain results of similar quality in significantly less
computational time for large $n$. To our knowledge, this is the first
usage of Goemans' compact formulation of the permutahedron in a convex
optimization problem. We also introduce a simpler regularization
scheme for this convex formulation of the 2-SUM problem that yields
good empirical results.
\end{abstract}

\section{Introduction} \label{sec.introduction}

A typical workflow for converting a discrete optimization problem over
the set of permutations of $n$ objects into a continuous relaxation is
as follows: (1) use permutation matrices to represent permutations;
(2) relax to the convex hull of the set of permutation matrices ---
the Birkhoff polytope; (3) relax other constraints to ensure
convexity/continuity. Examples of this procedure appear in
\cite{Fiori2013,Fogel2013}.  Representation of the Birkhoff polytope
requires $\Theta(n^2)$ variables, significantly more than the $n$
variables required to represent the permutation directly. The increase
in dimension is unappealing, especially if we are only interested in
optimizing over permutation vectors, as opposed to permutations of a
more complex object, such as a graph. The obvious alternative of using
a relaxation based on the convex hull of the set of permutations (the
permutahedron) is computationally infeasible, because the
permutahedron has exponentially many facets (whereas the Birkhoff
polytope has only $n^2$ facets). We can achieve a better trade-off
between the number of variables and facets by considering the
\emph{extension complexity} of the permutahedron, which is the minimum
number of linear inequalities required to describe a polytope that can
be linearly projected onto the permutahedron. Goemans
\cite{Goemans2010} recently proved that the extension complexity of
the permutahedron is $\Theta(n \log n)$ by describing how
\emph{sorting networks} can be used to construct such polytopes with
as few as $\Theta(n \log n)$ facets, and by providing a matching lower
bound. In this paper, we use a relaxation based on these polytopes,
which we call ``sorting network polytopes.'' When optimizing over the
set of permutations of a linear vector, we can use this tighter
formulation to obtain results similar to those obtained with the 
Birkhoff polytope, but in significantly less time, for large values 
of $n$.

\begin{table}[t] 
  \centering
  {
  \begin{tabular}{|c|c|c|}
    \hline Polytope         & Dimensions & Facets \\
    \hline Permutahedron    & $n - 1$                 & $2^n - 2$ \\
    \hline Birkhoff         & $n^2 - 2n + 1$               & $n^2$  \\
    \hline Sorting Network  & $\Theta(n \log n)$  & $\Theta(n \log n)$  \\
    \hline
  \end{tabular}
  }
  \caption{Comparison of three different polytopes. The figures
    listed under `Dimensions' are the dimensions for each of the polytopes as
    opposed to the dimension of the space they live in. The asymptotic figures
    listed for the sorting network polytope represents the best achievable 
    bounds.}
  \label{tbl:polytope_comparison}
\end{table}

We apply the sorting network polytope to the \emph{noisy seriation
problem}, defined as follows. Given a noisy similarity matrix $A$,
recover a symmetric row/column ordering of $A$ for which the entries
generally decrease with distance from the diagonal. Fogel
et~al.~\cite{Fogel2013} introduced a convex relaxation of the 2-SUM
problem to solve the noisy seriation problem. They proved that the
solution to the 2-SUM problem recovers the exact solution of the
seriation problem in the ``noiseless'' case (the case in which an
ordering exists that ensures monotonic decrease of similarity measures
with distance from the diagonal). They further show that the
formulation allows for the incorporation of side information about the
ordering, and is more robust to noise than a spectral formulation of
the 2-SUM problem described by Atkins et~al.~\cite{Atkins1998}. The
formulation in \cite{Fogel2013} makes use of the Birkhoff polytope. We
propose instead a formulation based on the sorting network polytope.
Performing convex optimization over the sorting network polytope
requires different formulation and solution techniques from those
described in \cite{Fogel2013}. In addition, we describe a new
regularization scheme, applicable both to our formulation and that of
\cite{Fogel2013}, that is more natural for the 2-SUM problem and has
good practical performance.

The paper is organized as follows. We begin by describing polytopes
for representing permutations in Section~\ref{sec:permutahedron}. In
Section~\ref{sec:convex}, we introduce the seriation problem and the
2-SUM problem and describe two continuous relaxations for the latter,
one of which uses the sorting network polytope. In the process, we
derive a new and simple regularization scheme for strengthening the
relaxations. Issues that arise in using the sorting network polytope
are discussed in Section~\ref{sec:practical}. Here we describe methods
for obtaining a permutation from a point in the permutahedron, methods
that are useful for solving the optimization problem efficiently, and
strengthening of the formulation by using side information about
ordering. In Section~\ref{sec:exp}, we provide experimental results
showing the effectiveness of our approach. The appendix includes the
proofs of the results covered in the course of the paper. It also
describes an efficient algorithm for taking a conditional gradient
step for the convex formulation in the special case in which the
formulation contains no side information, and gives some additional 
computational results.

\section{Permutahedron, Birkhoff Polytope, and Sorting Networks}
\label{sec:permutahedron}

In this section, we introduce relevant notation and review the
polytopes used in the rest of the paper.

We use $n$ throughout the paper to refer to the length of the
permutation vectors. $\pi_{I_n} = (1,2,\dotsc,n)^T$ denotes the
identity permutation. (When the size $n$ can be inferred from the
context, we write the identity permutation as $\pi_I$.)  $\cP^n$
denotes the set of all permutations vectors of length $n$. We use $\pi
\in \cP^n$ to denote a generic permutation, and denote its components
by $\pi(i)$, $i=1,2,\dotsc,n$. We use $\ones$ to denote the vector of
length $n$ whose components are all $1$.
\begin{definition}
  The \emph{permutahedron} $\cPH^n$, the convex hull of $\cP^n$, is 
  \[
    \left\{ x \in \bR^n \:\middle|\: \sum_{i=1}^n x_i =
      \frac{n(n+1)}{2},
    \sum_{i \in S} x_i \leq 
    \sum_{i=1}^{|S|} (n + 1 -i) \mbox{ for all }S\subset [n] \right\}.
  \]
\end{definition}
The permutahedron $\cPH^n$ has $2^n - 2$ facets, which prevents us
from using it in optimization problems directly. (We should note
however that the permutahedron is a submodular polyhedron and hence
admits efficient algorithms for certain optimization problems.)
Relaxations are commonly derived from the set of permutation matrices
(the set of $n \times n$ matrices containing zeros and ones, with a
single one in each row and column) and its convex hull instead.

\begin{definition}
  The convex hull of the set of $n \times n$ permutation matrices is
  the \emph{Birkhoff polytope} $\cB^n$, which is the set of all
  doubly-stochastic $n \times n$ matrices:
  \[
    \{ X \in \bR^{n \times n} \mid X \geq 0, X\ones = \ones, 
      X^T\ones = \ones \}.
  \]
\end{definition}

The Birkhoff polytope has been widely used in the machine learning and
computer vision communities for various permutation problems (see for
example \cite{Fogel2013}, \cite{Fiori2013}). The permutahedron can be
represented as the projection of the Birkhoff polytope from $\bR^{n
\times n}$ to $\bR^n$ by $x_i = \sum_{j=1}^n j \cdot X_{ij}$.  The
Birkhoff polytope is sometimes said to be an \emph{extended
formulation} of the permutahedron.

A natural question to ask is whether a more compact extended
formulation exists for the permutahedron. Goemans~\cite{Goemans2010}
answered this question in the affirmative by constructing an extended
formulation with $\Theta(n\log n)$ constraints and variables, which is
optimal up to constant factors. His construction is based on
\emph{sorting networks}, a collection of wires and binary comparators
that sorts a list of numbers.  Figure~\ref{fig:sorting_network}
displays a sorting network on $4$ variables. (See \cite{Cormen2001}
for further more information on sorting networks.)

\begin{figure} 
  \begin{center}
    \begin{tabular}{cc}
      \parbox[c]{6.5cm}{\includegraphics[width=60mm]{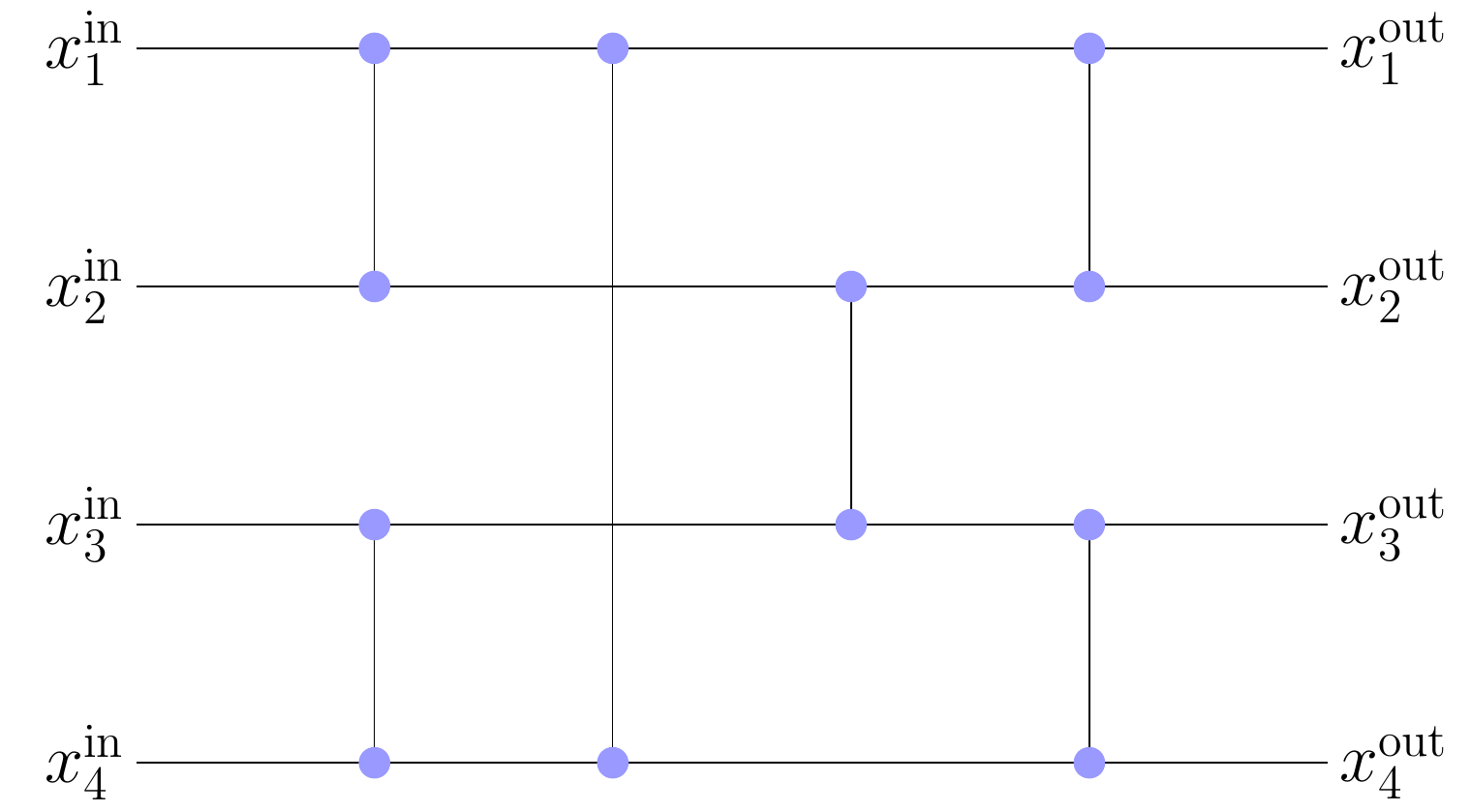}}
      &\parbox[c]{5.0cm}{\includegraphics[width=50mm]{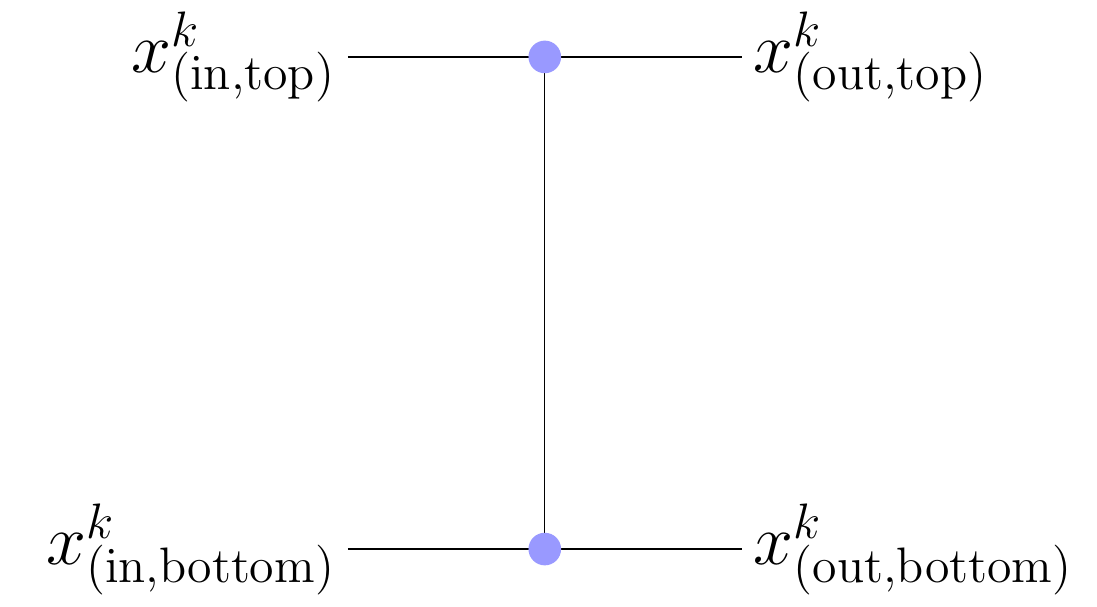}}
    \end{tabular}
  \end{center}
  \caption{\label{fig:sorting_network}A bitonic sorting network on 4
    variables (left) and the $k$-th comparator (right). The input to
    the sorting network is on the left and the output is on the
    right. At each comparator, we take the two input values and sort
    them such that the smaller value is the one at the top in the
    output. Sorting takes place progressively as we move from left to
    right through the network, sorting pairs of values as we encounter
    comparators.}
\end{figure}

Given a sorting network on $n$ inputs with $m$ comparators (we will
subsequently always use $m$ to refer to the number of comparators), an
extended formulation for the permutahedron with $O(m)$ variables and
constraints can be constructed as follows \cite{Goemans2010}.
Referring to the notation in the right subfigure in
Figure~\ref{fig:sorting_network}, we introduce a set of constraints
for each comparator $k=1,2,\dotsc,m$ to indicate the relationships
between the two inputs and the two outputs of each comparator:
\begin{align} \label{eqn:comparator}
  \begin{split}
    x^{k}_{\text{(in, top)}} + x^{k}_{\text{(in, bottom)}} &= 
    x^{k}_{\text{(out, top)}} + x^{k}_{\text{(out, bottom)}} \\
    x^{k}_{\text{(out, top)}} &\leq x^{k}_{\text{(in, top)}}\\
    x^{k}_{\text{(out, top)}} &\leq x^{k}_{\text{(in, bottom)}}.
  \end{split}
\end{align}
Note that these constraints require the sum of the two inputs to be
the same as the sum of the two outputs, but the inputs can be closer
together than the outputs.  Let $x^\text{in}_i$ and $x^\text{out}_i$,
$i=1,2,\dotsc,n$ denote the $x$ variables corresponding to the $i$th
input or output to the entire sorting network, respectively. We
introduce the additional constraints
\begin{align} \label{eqn:output_wires}
  x^\text{out}_i = i, \mbox{ for } i \in [n].
\end{align}
The details of this construction depend on the particular choice of
sorting network (see Section~\ref{sec:practical}), but we will refer
to it generically as the \emph{sorting network polytope} $\cSN^n$.
Each element in this polytope can be viewed as a concatenation of two
vectors: the subvector associated with the network inputs
$x^{\text{in}} = (x^\text{in}_1, x^\text{in}_2, \dotsc,
x^\text{in}_n)$, and the rest of the coordinates $x^{\text{rest}}$,
which includes all the internal variables as well as the outputs.

\begin{theorem}[Goemans \cite{Goemans2010}] \label{thm:goemans}
  Given the construction above, the set $\{ x^{\text{in}} \:|\;
  (x^{\text{in}}, x^{\text{rest}}) \in \cSN^n \}$ is the permutahedron
  $\cPH^n$.
\end{theorem}
In fact, the proof of Theorem \ref{thm:goemans} shows that this
construction can be used to describe the convex hull of the
permutations of any single vector.  Let $v$ be a 
monotonically-increasing vector. Then, if we replace the constraints
\eqref{eqn:output_wires} with
\begin{align} \label{eqn:output_wires_general}
  x^\text{out}_i = v_i, \mbox{ for } i \in [n]
\end{align}
the set $\{ x^{\text{in}} \} $ now corresponds to the convex hull of
all permutations of $v$. We include a simple alternative proof of this
fact (and by extension Theorem \ref{thm:goemans}) in
Appendix~\ref{app:proofs}.

\section{Convex Relaxations of 2-SUM via Sorting Network Polytope} 
\label{sec:convex}

In this section we will briefly describe the seriation problem, and
some of the continuous relaxations of the combinatorial 2-SUM problem
that can be used to solve this problem.

\subsection{The Noiseless Seriation Problem}

The term \emph{seriation} generally refers to data analysis techniques
that arrange objects in a linear ordering in a way that fits available
information and thus reveals underlying structure of the system
\cite{Liiv2010}. We adopt here the definition of the \emph{seriation
  problem} from \cite{Atkins1998}.  Suppose we have $n$ objects
arranged along a line, and a similarity function that increases with
distance between objects in the line. The similarity matrix is the $n
\times n$ matrix whose $(i, j)$ entry is the similarity measure
between the $i$th and $j$th objects in the linear arrangement. This
similarity matrix is a \emph{R-matrix}, according to the following
definition.
\begin{definition}
  A symmetric matrix $A$ is a \emph{Robinson matrix} (\emph{R-matrix})
  if for all points $(i, j)$ where $i > j$, we have $A_{ij} \leq \min
  (A_{(i-1)j}, A_{i(j+1)})$.  A symmetric matrix $A$ is a \emph{pre-R
    matrix} if $\Pi^T A \Pi$ is R for some permutation $\Pi$.
\end{definition}
In other words, a symmetric matrix is a R-matrix if the entries are
nonincreasing as we move away from the diagonal in either the
horizontal or vertical direction. The goal of the \emph{noiseless
seriation problem} is to recover the ordering of the variables along
the line from the pairwise similarity data, which is equivalent to
finding the permutation that recovers an R-matrix from a pre-R-matrix.

The seriation problem was introduced in the archaeology literature
\cite{Robinson1951}, and has applications across a wide range of areas
including clustering \cite{Ding2004}, shotgun DNA sequencing
\cite{Fogel2013}, and taxonomy \cite{Sokal1963}.  R-matrices are
useful in part because of their relation to the \emph{consecutive-ones
property} in a matrix of zeros and ones, where the ones in each
column form a contiguous block. 
A matrix $M$ with the consecutive-ones property gives rise to a
R-matrix $MM^T$, so the matrix $\Pi^T M$ with rows permuted by $\Pi$
leads to a pre-R-matrix $\Pi^T M M^T\Pi$.

\subsection{Noisy Seriation, 2-SUM and Continuous Relaxations}

Given a binary symmetric matrix $A$, the 2-SUM problem can be
expressed as the following.
\begin{align} \label{eqn:2sum}
  \min_{\pi \in \cP^n}\:\: & \sum_{i=1}^n 
    \sum_{j=1}^n A_{ij} (\pi(i) - \pi(j))^2.
\end{align}
A slightly simpler but equivalent formulation, defined via the
Laplacian $L_A = \mbox{diag}(A \ones ) - A$, is
\begin{align} \label{eqn:2sumlap}
    \min_{\pi \in \cP^n}\:\: \pi^T L_A \pi.
\end{align}
The seriation problem is closely related to the combinatorial 2-SUM
problem, and Fogel et~al.\ \cite{Fogel2013} proved that if $A$ is a
pre-$R$-matrix such that each row/column has unique entries, then the
solution to the 2-SUM problem also solves the noiseless seriation
problem. In another relaxation of the 2-SUM problem, Atkins et~al.\
\cite{Atkins1998} demonstrate that finding the second smallest
eigenvalue, also known as the Fiedler value, solves the noiseless
seriation problem.  Hence, the 2-SUM problem provides a good model for
the \emph{noisy} seriation problem, where the similarity matrices are
close to, but not exactly, pre-R matrices.

The 2-SUM problem is known to be $NP$-hard \cite{George1997}, so we
would like to study efficient relaxations. We describe below two
continuous relaxations that are computationally practical. (Other
relaxations of these problems require solution of semidefinite
programs and are intractable in practice for large $n$.)

The spectral formulation of \cite{Atkins1998} seeks the Fiedler value
by searching over the space orthogonal to the vector $\ones$, which is
the eigenvector that corresponds to the zero eigenvalue. The Fiedler
value is the optimal objective value of the following problem:
\begin{align} \label{eqn:fiedler}
  \min_{y\in\bR^n}\:\:  y^T  L_A  y
  \quad\st\quad  y^T\ones=0, \;\; \|y\|_2=1.
\end{align}
This problem is non-convex, but its solution can be found efficiently
from an eigenvalue decomposition of $L_A$.  With Fiedler vector $y$,
one can obtain a candidate solution to the 2-SUM problem by picking
the permutation $\pi \in \cP^n$ to have the same ordering as the
elements of $y$. Another perspective is given by \cite{Vuokko2010},
who prove that the spectral solution minimizes the Frobenius norm of
the strongest principal component of a particular derived matrix, and
minimizing the Frobenius norm of this matrix is equivalent to
minimizing the sums of consecutive zeros between the first and last
non-zero entry of a 0-1 matrix.  We show in Appendix~\ref{sec:sf2}
that the spectral formulation \eqref{eqn:fiedler} is a continuous
relaxation of the 2-SUM problem \eqref{eqn:2sumlap}.

The second relaxation of \eqref{eqn:2sumlap}, described by Fogel
et~al.\ \cite{Fogel2013}, makes use of the Birkhoff polytope $\cB^n$.
The basic version of the formulation is
\begin{align} \label{eqn:fogel_basic}
  \min_{\Pi \in \cB^n}\:\: & \pi_I^T \Pi^T L_A \Pi \pi_I,
\end{align} 
(recall that  $\pi_I$ is the identity permutation $(1, 2, \dotsc,
n)^T$), which is a convex quadratic program.  Fogel et~al.\ augment
and enhance this formulation as follows.
\begin{itemize} 
  \item Introduce a ``tiebreaking'' constraint $e_1^T \Pi \pi_I + 1 
    \leq e^T_n\Pi \pi_I$ to resolve ambiguity about the direction of
    the ordering, where $e_k = (0,\dotsc,0,1,0,\dotsc,0)^T$ with the
    $1$ in position $k$.
  \item Average over several perturbations of $\pi_I$ to improve
    robustness of the solution.
  \item Add a penalty to maximize the Frobenius norm of the matrix 
    $\Pi$, which pushes the solution closer to a vertex of the
    Birkhoff polytope.
  \item Incorporate additional ordering constraints of the form 
    $x_i - x_j \leq \delta_k$, to exploit prior knowledge about the
    ordering.
\end{itemize}
With these modifications, the problem to be solved is
\begin{align} \label{eqn:fogel_full}
  \min_{\Pi \in \cB^n}\:\: \frac{1}{p} 
    \text{Trace}( Y^T \Pi^T L_A \Pi Y) - \frac{\mu}{p} \| P\Pi \|^2_F 
  \quad\st\quad  D \Pi \pi_I \leq \delta,
\end{align}
where each column of $Y \in \bR^{n\times p}$ is a slightly perturbed
version of a permutation,\footnote{In \cite{Fogel2013}, each column of
$Y$ is said to contain a perturbation of $\pi_I$, but in a response to
referees of their paper, the authors say that they used sorted uniform
random vectors instead in the revised version.} $\mu$ is the
regularization coefficient, the constraint $D\Pi \pi_I \leq \delta$
contains the ordering information and tie- breaking constraints, and
the operator $P = I -\frac{1}{n} \ones \ones^T$ is the projection of
$\Pi$ onto elements orthogonal to the all-ones matrix. The
penalization is applied to $\| P\Pi \|^2_F$ rather than to $\| \Pi
\|^2_F$ directly, thus ensuring that the program remains convex if the
regularization factor is sufficiently small (for which a sufficient
condition is $\mu < \lambda_2 (L_A) \lambda_1 (YY^T)$). We will refer
to this regularization scheme as the \emph{matrix-based
regularization}, and to the formulation \eqref{eqn:fogel_full} as the
\emph {matrix-regularized Birkhoff-based convex formulation}.

Figure~\ref{fig:spectral_vs_fogel} illustrates the permutahedron and
the solutions produced by the methods described above. The spectral
method returns good solutions when the noise is low and it is
computationally efficient since there are many fast algorithms and
software for obtaining selected eigenvectors. However, the Birkhoff-
based convex formulation can return a solution that is significantly
better in situations with high noise or significant additional
ordering information.  For the rest of this section, we will focus on
the convex formulation.

\subsection{A Compact Convex Relaxation via the 
Permutahedron/Sorting Network Polytope and a New Regularization
Scheme}

\begin{figure} 
  \begin{center}
    \begin{tabular}{cc}
      \parbox[c]{6.5cm}{\includegraphics[width=55mm]{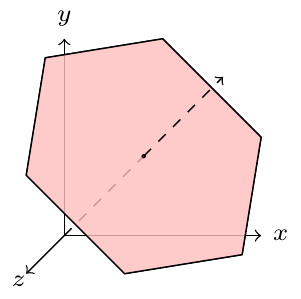}}
      &\parbox[c]{6.5cm}{\includegraphics[width=60mm]{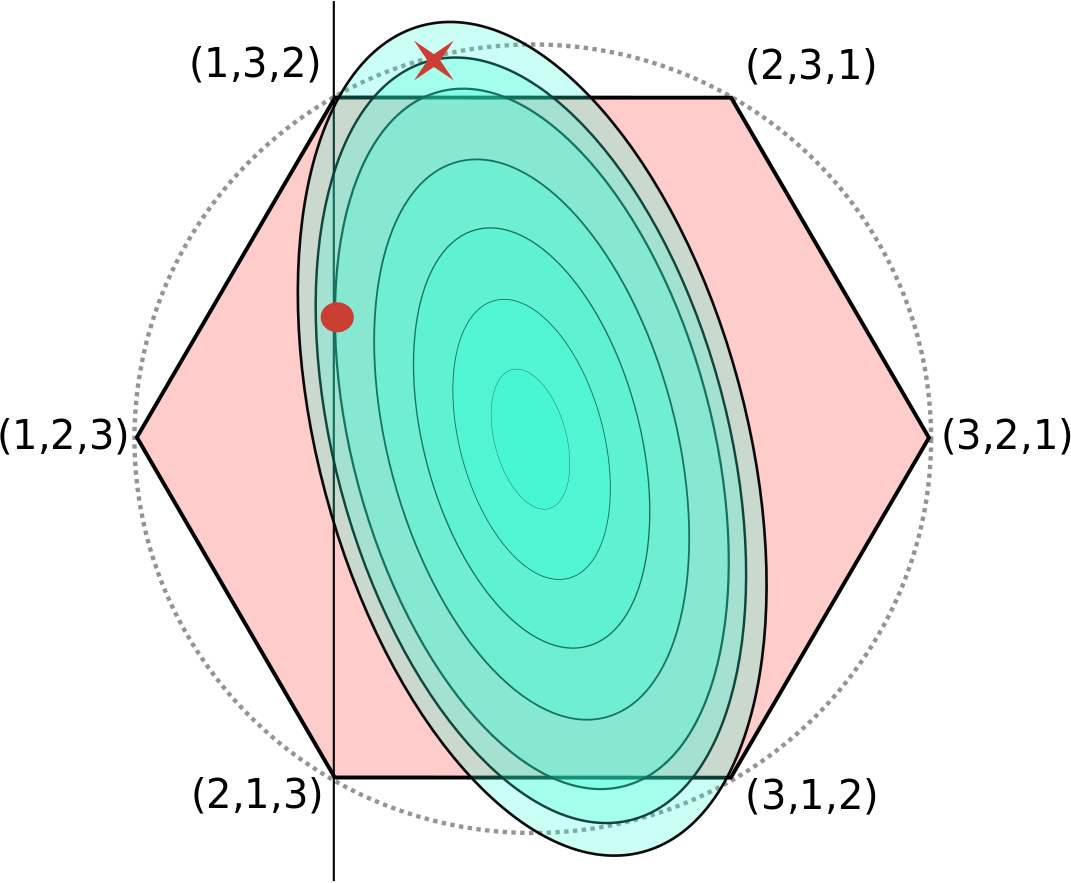}}
    \end{tabular}
  \end{center}
  \caption{A geometric interpretation of spectral and convex
    formulation solutions on the 3-permutahedron. The left image shows
    the 3-permutahedron in 3D space and the dashed line shows the
    eigenvector $\ones$ corresponding to the zero eigenvalue.  The
    right image shows the projection of the 3-permutahedron along the
    trivial eigenvector together with the elliptical level curves of
    the  objective function $y^T L_A y$. Points on the
    circumscribed circle have an $\ell_2$-norm equal to that of a
    permutation, and the objective is minimized over this
    circle at the point denoted by a cross.  The vertical line in the
    right figure enforces the tiebreaking constraint that $1$ must
    appear before $3$ in the ordering; the red dot indicates the
    minimizer of the objective over the resulting triangular
    feasible region. Without the
    tiebreaking constraint, the minimizer is at the center of the
    permutahedron.}
    \label{fig:spectral_vs_fogel}
\end{figure}

We consider now a different relaxation forthe 2-SUM problem
\eqref{eqn:2sumlap}. Taking the convex hull of $\cP^n$ directly, we
obtain
\begin{align} \label{eqn:perm_basic}
  \min_{x \in \cPH^n}\:\: & x^T L_A x.
\end{align}
This is essentially a permutahedron-based version of
\eqref{eqn:fogel_basic}. In fact, two problems are equivalent, except
that formulation (\ref{eqn:perm_basic}) is more compact when we
enforce $x \in \cPH$ via the sorting network constraints
\[ 
x \in \{ x^{\text{in}} \:|\; (x^{\text{in}}, x^{\text{rest}}) \in
\cSN^n \},
\]
where $\cSN^n$ incorporates the comparator constraints and the output
constraints \eqref{eqn:output_wires}.  This formulation can be
enhanced and augmented in a similar fashion to
(\ref{eqn:fogel_basic}). The tiebreaking constraint for this
formulation can be expressed simply as $x_1 + 1 \leq x_n$, since
$x^{\text{in}}$ consists of the subvector $(x_1,x_2,\dotsc,x_n)$.  (In
both (\ref{eqn:perm_basic}) and (\ref{eqn:fogel_basic}), having at
least one additional constraint is necessary to remove the trivial
solution given by the center of the permutahedron or Birkhoff
polytope; see Figure~\ref{fig:spectral_vs_fogel}.) This constraint is
the strongest inequality that will not eliminate any permutation
(assuming that a permutation and its reverse are equivalent); we
include a proof of this fact in
Appendix~\ref{sec:tiebreak_characterization}.

It is also helpful to introduce a penalty to force the solution $x$ to
be closer to a permutation, that is, a vertex of the permutahedron. To
this end, we introduce a \emph{vector-based regularization scheme}.
The following statement about the norms of permutations is an
immediate consequence of strict convexity of norms.

\begin{proposition} \label{prop:perm_norm}
  Let $v \in \bR^n$, and consider the set $X$ consisting of the convex
  hull of all permutations of $v$. Then, the points in $X$ with the
  highest $\ell_p$ norm, for $1<p<\infty$, are precisely the
  permutations of $v$.
\end{proposition}

It follows that adding a penalty to encourage $\| x \|_2$ to be large
might improve solution quality. However, directly penalizing the
negative of the $2$-norm of $x$ would destroy convexity, since $L_A$
has a zero eigenvalue. Instead we penalize $Px$, where $P = I -
\frac{1}{n} \ones \ones^T$ projects onto the subspace orthogonal to
the trivial eigenvector $\ones$. (Note that this projection of the
permutahedron still satisfies the assumptions of
Proposition~\ref{prop:perm_norm}.) When we include a penalty on $\|
Px\|_2^2$ in the formulation \eqref{eqn:perm_basic} along with side
constraints $Dx \le \delta$ on the ordering, we obtain
\begin{align*}
  \min_{x \in \cPH^n}\:\: x^T L_A  x  - \mu\|Px\|_2^2 
  \quad\st\quad Dx \leq \delta,
\end{align*}
which can be written as 
\begin{align} \label{eqn:perm_full}
  \min_{x \in \cPH^n}\:\:  x^T (L_A - \mu P) x 
  \quad \st \quad Dx \leq \delta.
\end{align}
Vector-based regularization decreases all nonzero eigenvalues of $L_A$
by the constant $\mu$ while keeping the same eigenvectors. This means
that (\ref{eqn:perm_full}) is convex if $\mu$ is smaller than the
Fiedler value. This regularization procedure can be thought of as a
penalty term corresponding to the constraints in the Fiedler
formulation. We will refer to this formulation as the
\emph{regularized permutahedron-based convex formulation}.

Vector-based regularization can also be incorporated into the
Birkhoff-based convex formulation. Consider the following corollary of
Proposition~\ref{prop:perm_norm} that shows that instead of maximizing
the $\| P\Pi \|_2^2$ term in formulation \eqref{eqn:fogel_full} to
force the solution to be closer to a permutation, we could maximize
$\| P \Pi Y \|_2^2$:
\begin{corollary} \label{cor:perm_norm}
  Let $Y \in \bR^{n \times p}$ be a matrix where every column has no
  repeated elements, and consider the set $X = \{P \Pi Y \:|\: \Pi \in
  \cB^n \}$ where $P=I - \frac{1}{n} \ones \ones^T$. The elements with
  the highest Frobenius norm in $X$ are given by the set $\{ P \Pi Y
  \:|\: \Pi \text{ is a permutation matrix} \}$.
\end{corollary}
The vector-regularized version of \eqref{eqn:fogel_basic} with side
constraints can be written as follows:
\begin{align} \label{eqn:fogel_full_vec}
  \min_{\Pi \in \cB^n}\:\: & \frac{1}{p} 
    \Tr( Y^T \Pi^T (L_A - \mu P) \Pi Y) 
  \quad\st\quad D\Pi \pi_I \leq \delta.
\end{align}
We refer to this formulation as the \emph{vector-regularized
  Birkhoff-based convex formulation}.  We note that when we let $Y =
\pi_I = (1,2,\dotsc,n)^T$, the solution of the regularized
permutahedron-based convex formulation and the optimal value
$\Pi^\ast\pi_I$ (where $\Pi^\ast$ is the solution of
\eqref{eqn:fogel_full_vec}) are the same. Hence, we can view the
regularized permutahedron-based convex formulation as a compact way of
encoding the special case of \eqref{eqn:fogel_full_vec} for which $p =
1$.

Vector-based regularization is in some sense more natural than the
regularization in \eqref{eqn:fogel_full}. It acts directly on the set
that we are optimizing over, rather than an expanded set. (Different
elements $\Pi_1$ and $\Pi_2$ of the Birkhoff polytope may yield the
same permutation vector: $\Pi_1 \pi_I = \Pi_2 \pi_I$.) In addition,
the objective of \eqref{eqn:fogel_full_vec} remains convex provided
that $\mu < \lambda_2(L_A)$, a significantly looser constraint that
for the matrix-based regularization scheme, allowing for stronger
regularization.

The regularized permutahedron-based convex formulation is a convex QP
with $O(m)$ variables and constraints, where $m$ is the number of
comparators in its sorting network, while the Birkhoff-based one is a
convex QP with $O(n^2)$ variables. In addition, the use of the 
vector-based regularization scheme in formulations 
\eqref{eqn:perm_full} and \eqref{eqn:fogel_full_vec} allows standard
convex QP solvers (such as those based on interior-point methods) to
be applied to both formulations. The one feature in the Birkhoff-based
formulations that the permutahedron-based formulations do not have is
the ability to average the solution over multiple vectors by choosing
dimension $p > 1$ for the matrix $Y \in \bR^{n \times p}$. However,
our experiments suggested that the best solutions were obtained for $p
= 1$, so this consideration was not important in practice.

\section{Key Implementation Issues} \label{sec:practical}

Having described the permutahedron-based convex formulation, we now
discuss the important choices to be made in constructing the
relaxation and in extracting a suitable permutation from the solution
of the relaxation. We also briefly discuss algorithms for solving
these formulations, and possible strengthening of the formulations.

\paragraph{Choice of Sorting Network.}

There are numerous possible choices of the sorting network, from which
the constraints in formulation (\ref{eqn:perm_full}) are derived. The
asymptotically most compact option is the AKS sorting network, which
contains $\Theta(n \log n)$ comparators. This network was introduced
in \cite{Ajtai1983} and subsequently improved by others, but is
impractical because of its difficulty of construction and the large
constant factor in the complexity expression. We opt instead for more
elegant networks with slightly worse asymptotic complexity.
Batcher~\cite{Batcher1968} introduced two sorting networks with
$\Theta(n \log^2 n)$ size --- the odd-even sorting network and the
bitonic sorting network --- that are popular and
practical. Generation of the constraints that describe the sorting
network polytope can be performed with a simple recursive algorithm
that runs in $\Theta(n \log^2 n)$ time. The bitonic sorting network
gave good performance in our implementations.

\paragraph{Obtaining Permutations from a Point in the Permutahedron.} 

Solution of the permutation-based relaxation yields a point $x$ in the
permutahedron, but we need techniques to convert this point into a
valid permutation, which is a candidate solution for the 2-SUM problem
\eqref{eqn:2sum}.  The most obvious recovery technique is to choose
this permutation $\pi$ to have the same ordering as the elements of
$x$, that is, $x_i < x_j$ implies $\pi(i) < \pi(j)$, for all $i,j \in
\{1,2,\dotsc,n\}$. We could also sample multiple permutations, by
applying Gaussian noise to the components of $x$, prior to taking the
ordering to produce $\pi$. (In our experiments we added i.i.d.
Gaussian noise with variance $0.5$ to each element of $x$ before
ordering.) The 2-SUM objective \eqref{eqn:2sum} can be evaluated for
each permutation so obtained, with the best one being reported as the
overall solution. This randomized recovery procedure can be repeated
many times, as it is quite inexpensive. Fogel et~al.~\cite{Fogel2013}
propose a somewhat different permutation sampling procedure for
matrices $\Pi$ in the Birkhoff polytope, obtaining a permutation by
sorting the vector $\Pi v$, where $v$ is a sorted random vector. We
experimented too with decomposition-based methods, in which $x$ is
expressed as a convex combination of permutations $\sum_{i=1}^{n+1}
\lambda_i \pi^i$, a decomposition that can be found efficiently using
an optimal $O(n^2)$ algorithm \cite{Yasutake2011}. We evaluated some
or all of the spanning permutations by treating the permutations as
permutation matrices and applying the sampling procedure from
\cite{Fogel2013}, but this approach yielded weaker solutions in
general.

\paragraph{Solving the Convex Formulation.}
On our test machine using the Gurobi interior point solver, we were
able to solve instances of the permutahedron-based convex formulation
(\ref{eqn:perm_full}) of size up to around $n=10000$. As in
\cite{Fogel2013}, first-order methods can be employed when the scale
is larger. 

\paragraph{Strengthening Side Structural Information.}

The constraint set for our permutation-based relaxation is a
polyhedron defined by the sorting network constraints and the side
constraints. A much stronger relaxation would be obtained if we could
identify the feasible points that correspond to actual permutations,
and optimize over the convex hull of these points. (As an example, the
set of two constraints $x_1 + 1 \leq x_2$ and $x_1 + 1 \leq x_3$ over
the 3-permutahedron is satisfied by only two permutations.) However,
identifying this convex hull is computationally difficult in general.
For example, the convex hull of the set of permutations that satisfy a
set of constraints of the form $x_i + 1 \leq x_j$ corresponds
precisely to the``Single Machine Scheduling Polytope with Scheduling
Constraints and Unit Costs'', and since it is $NP$-hard to optimize a
linear function over this polytope, computing a linear description is
also difficult. Many sets of valid inequalities have been derived for
the corresponding scheduling polytope (see \cite{Queyranne94}), and it
will be interesting future work to understand the trade-offs between
the tighter solutions that could be obtained by incorporating these
valid inequalities and the run time from computing these inequalities.

\section{Experiments} \label{sec:exp}

We compare the run time and solution quality of algorithms on the two
classes of convex formulations --- Birkhoff-based and
permutahedron-based --- with various parameters. Summary results are
presented in this section, with additional results appearing in
Appendix~\ref{sec:additional_experiments}.

\subsection{Experimental Setup} 

\label{sec:setup}

The experiments were run on an Intel Xeon X5650 (24 cores @ 2.66Ghz)
server with 128GB of RAM in MATLAB 7.13, CVX 2.0
(\cite{cvx},\cite{gb08}), and Gurobi 5.5 \cite{gurobi}. We tested five
formulation-algorithm-implementation variants, as follows.
\begin{enumerate}
\item Spectral method using the MATLAB {\tt eigs} function.
\item Frank-Wolfe algorithm \cite{Frank1956} on the Birkhoff-based
  convex formulations using MATLAB/CVX/Gurobi to solve the linear
  subproblems.
\item MATLAB/Gurobi on the permutahedron-based convex formulation.
\item MATLAB/Gurobi on the Birkhoff-based convex formulation with $p =
  1$ (that is, \eqref{eqn:fogel_full} with $Y = \pi_I$).
\item Experimental MATLAB code provided to us by the authors of
  \cite{Fogel2013} implementing FISTA, for solving the
  matrix-regularized Birkhoff-based convex formulation
  \eqref{eqn:fogel_full}, with projection steps solved using block
  coordinate ascent on the dual problem. This is the current
  state-of-the-art algorithm for large instances of the Birkhoff-based
  convex formulation; we refer to it as RQPS (for ``Regularized QP for
  Seriation'').
\end{enumerate}

We report run time data using wall clock time reported by Gurobi, and
MATLAB timings for RQPS, excluding all preprocessing time.

We used the bitonic sorting network by Batcher~\cite{Batcher1968} for
experiments with the permutahedron-based formulation. For consistency
between approaches, we used the procedure described in
Section~\ref{sec:practical} for collapsing relaxed solutions to
permutations in all cases. For Birkhoff matrices, this means that we
applied random Gaussian perturbations (drawn i.i.d.\ from $N(0,0.5)$)
to elements of the vector $\Pi^* \pi_I$, where $\Pi^*$ is the solution
of the Birkhoff relaxation. We picked the default parameter settings
for the Gurobi interior-point solver except for the convergence
tolerance, which is experiment-dependent. We provide more details on
algorithm parameters specific to the different variants below.

We vary the number of ordering constraints added for each experiment.
Each ordering constraint is of the form $x_i + \pi(j) - \pi(i) \leq
x_j$, where $\pi$ is the (known) permutation that recovers the
original matrix, and $i,j \in [n]$ is a pair randomly chosen but
satisfying $\pi(j) - \pi(i) > 0$. 

Besides run time, we used three other metrics to measure performance:
\begin{itemize}
\item the 2-SUM objective value of recovered matrix,
\item the total number of R-matrix constraints of the form $a_{ij}
  \leq a_{(i-1)j}$ or $a_{ij} \leq a_{i(j+1)}$ that are violated by
  the recovered matrix, and 
\item Kendall's $\tau$ score to measure how close the permutation
  obtained is to the permutation that recovers the original matrix.
\end{itemize} 

\subsection{Munsingen Dataset} 

We test the solution quality of the permutahedron-based formulations
on a standard problem in the seriation literature drawn from
archaeology. The Munsingen dataset, introduced by Hodson
\cite{Hodson1968} and manually re-ordered in \cite{Kendall1971}, was
used as a test problem in \cite{Fogel2013}.  It consists of a binary
matrix $M$ indicating the presence of each of 70 artifact types on 59
graves, where each artifact is presumed to be associated with a
particular time period. The goal is to order the graves by date. We
treat this matrix as a noisy consecutive-ones problem and minimize the
2-SUM objective over $MM^T$.

\begin{figure}[hbtp] 
  \begin{center}
    \begin{tabular}{ccc}
      \parbox[c]{56mm}{\includegraphics[width=45mm]{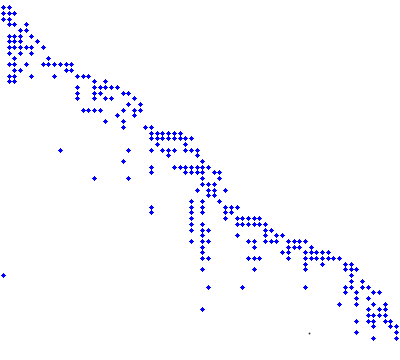}}
      \parbox[c]{56mm}{\includegraphics[width=45mm]{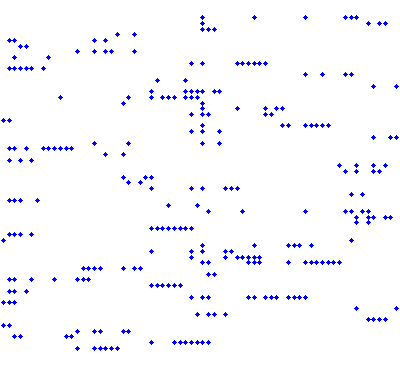}}
      \parbox[c]{56mm}{\includegraphics[width=45mm]{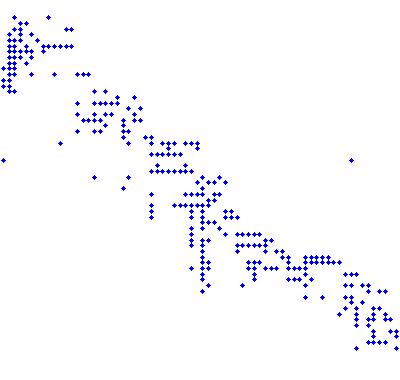}}
    \end{tabular}
  \end{center}
  \caption{Munsingen Dataset. Plot of the Munsingen data matrix $M$
    (left), the matrix with rows permuted randomly (center), and the
    matrix reordered according to the solution of the
    permutahedron-based convex formulation with 15 randomly chosen
    ordering constraints (right).  \label{fig:munsingen_illustrated}}
\end{figure}

Since we are interested in solution quality (rather than run time
performance) for this data set, we used the same algorithm across all
the different convex formulations, for consistency: the Frank-Wolfe
algorithm with step-size ${2}/{(k + 2)}$ (where $k$ is the iteration
number), terminating when the relative duality gap is less than one
percent. For the regularized permutahedron-based formulation, since it
gives the same solution as the vector-regularized Birkhoff-based
formulation when $p=1$, we simply apply the Frank-Wolfe algorithm to
the latter formulation. We checked that the solutions to the
permutahedron-based formulations obtained indirectly in this manner
and directly using the Gurobi interior-point solver on the
permutahedron-based formulations were similar in quality. We varied
the Birkhoff-based convex formulations in three ways. First, we chose
the $n \times p$ matrix $Y$ in \eqref{eqn:fogel_full} and
\eqref{eqn:fogel_full_vec} to have $p=n$ and $p=4n$ columns (we chose
$p \geq n$ to enable application of the matrix-based regularization
scheme), with each column in $Y$ chosen by sorting a vector whose
entries are independent uniformly distributed random variables in
$[0,1]$. (These choices of $p$ and $Y$ are consistent with those used
in \cite{Fogel2013}.)  Second, we tried both matrix and vector
regularization schemes (\eqref{eqn:fogel_full} and
\eqref{eqn:fogel_full_vec}, respectively).  Third, we varied the
regularization parameter to be $0\%$, $50\%$ and $90\%$ of the maximum
allowed for each respective scheme.

Table~\ref{tbl:munsingen_15} displays the results, averaged over ten
runs of each variant, each with a different randomly chosen set of 15
ordering constraints. The formulations with vector-based
regularization consistently outperform formulations without
regularization or with matrix-based regularization, both in the 2-SUM
objective and the R-score. The best 2-SUM objectives were obtained for
the permutahedron-based formulation. Using $p>1$ in the Birkhoff
formulation did not help. The spectral method, which could not make
use of any side information about the ordering, did not give
competitive results. Results obtained by using 38 ordering constraints
(instead of 15) are similar; we report these in
Appendix~\ref{sec:additional_experiments}.

\begin{table}[t] 
  \centering
  {
  \begin{tabular}{|c|c|cc|cc|cc|cc|}
    \hline
    Method & $p$ & Reg. Type & Level & 2-SUM & Std Err & R-score & Std Err & Kendall's $\tau$ & Std Err \\
    \hline
    Input &&&         & 77040 & 0 & 289.0 & 0.0 & 1.000 & 0.000 \\
    Spectral &&&      & 77806 & 0 & 295.0 & 0.0 & 0.755 & 0.001 \\
    \hline
    Permut. & $-$ & none   & $-$    & 72798 & 6008 & 306.6 & 14.9 & 0.863 & 0.016 \\
    Permut. & $-$ & vector & $50\%$ & 70515 & 5828 & 309.4 & 21.1 & 0.862 & 0.017 \\
    Permut. & $-$ & vector & $90\%$ & 69336 & 5422 & 302.8 & 22.5 & 0.867 & 0.016 \\
    \hline
    Birkhoff & $n$ & none & $-$      & 74111 & 6966 & 307.3 & 33.6 & 0.859 & 0.014 \\
    Birkhoff & $n$ & matrix & $50\%$ & 73718 & 6489 & 306.2 & 18.8 & 0.860 & 0.015 \\
    Birkhoff & $n$ & matrix & $90\%$ & 73810 & 6874 & 313.0 & 21.9 & 0.857 & 0.014 \\
    Birkhoff & $n$ & vector & $50\%$ & 71623 & 6100 & 297.6 & 16.5 & 0.860 & 0.019 \\
    Birkhoff & $n$ & vector & $90\%$ & 69898 & 5212 & 299.9 & 15.9 & 0.867 & 0.016 \\
    \hline
    Birkhoff & $4n$ & none   & $-$    & 73257 & 6713 & 311.6 & 16.7 & 0.856 & 0.015 \\
    Birkhoff & $4n$ & matrix & $50\%$ & 73624 & 6632 & 306.3 & 14.6 & 0.859 & 0.017 \\
    Birkhoff & $4n$ & matrix & $90\%$ & 73667 & 6589 & 305.9 & 21.5 & 0.858 & 0.013 \\
    Birkhoff & $4n$ & vector & $50\%$ & 70827 & 5582 & 297.4 & 21.9 & 0.862 & 0.009 \\
    Birkhoff & $4n$ & vector & $90\%$ & 69490 & 4927 & 291.8 & 15.9 & 0.868 & 0.016 \\
    \hline
  \end{tabular}
  }
  \caption{Munsingen Dataset: Performance with 15 ordering constraints.}
  \label{tbl:munsingen_15}
\end{table}

\subsection{Linear Markov Chain} 

The Markov chain reordering problem \cite{Fogel2013} involves
recovering the ordering of a simple Markov chain with Gaussian noise
from disordered samples. The Markov chain consists of random variables
$X_1, X_2, \dotsc, X_n$ such that $X_i = b X_{i-1} + \epsilon_i$ where
$\epsilon_i \sim N(0, \sigma^2)$. A sample covariance matrix taken
over multiple independent samples of the Markov chain is used as the
similarity matrix in the 2-SUM problem.

\begin{figure}[thbp]
  \begin{center}
    \begin{tabular}{ccc}
      \parbox[c]{56mm}{\includegraphics[width=45mm]{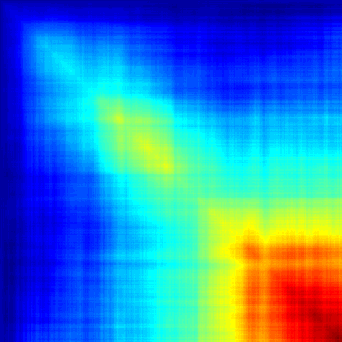}}
      \parbox[c]{56mm}{\includegraphics[width=45mm]{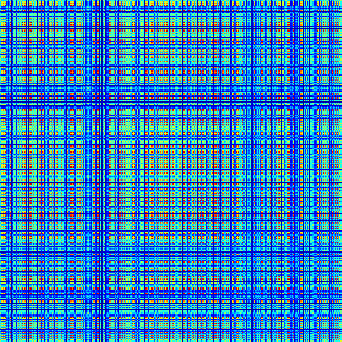}}
      \parbox[c]{56mm}{\includegraphics[width=45mm]{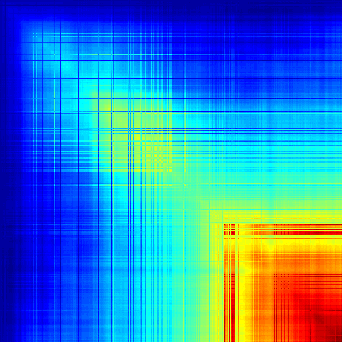}}
    \end{tabular}
  \end{center}
  \caption{Linear Markov Chain covariances.  Plot of a covariance
    matrix corresponding to 50 samples from a length 500 Markov chain
    (left), the same covariance matrix with rows and columns randomly
    permuted (center), and the reordered matrix obtained by solving
    the permutahedron-based convex formulation with 750 randomly
    chosen ordering constraints (right).
  \label{fig:markov_illustrated}}
\end{figure}

We use this problem for two different comparisons. First, we compare
the solution quality and running time of RQPS algorithm of
\cite{Fogel2013} with the Gurobi interior-point solver on the
regularized permutahedron-based convex formulation, to demonstrate the
performance of the formulation and algorithm introduced in this paper
compared with the prior state of the art. Second, we apply Gurobi to
both the permutahedron-based and Birkhoff-based formulations with $p =
1$, with the goal of discovering which formulation is more efficient in
practice.

For both sets of experiments, we fixed $b = 0.999$ and $\sigma=0.5$
and generate 50 sample sample chains to form a sample covariance
matrix.  We chose $n \in \{ 500, 2000, 5000 \}$ to test the scaling of
algorithm performance with dimension. For each $n$, we perform 10
independent runs, each based on a different set of samples of the
Markov chain (and hence a different sample covariance matrix).  Three
levels of side constraints were used --- $0.5n$, $n$, and $1.5n$ ---
chosen as described in Subsection~\ref{sec:setup}.
On initial tests, we observed that the choice of regularization scheme
did not significantly affect the performance, so we chose to use
vector-based regularization with parameter $\mu = 0.9\lambda_2(L_A)$
on all formulations throughout these two sets of experiments.

\paragraph{RQPS and the Permutahedron-Based Formulation.}

We now compare the RQPS code for the matrix-regularized Birkhoff-based
convex formulation \eqref{eqn:fogel_full} to the regularized
permutahedron-based convex formulation, solved with Gurobi. We fixed a
time limit (which differed according to the value of $n$) and ran the
RQPS algorithm until the limit was reached. At fixed time intervals,
we query the current solution and sample permutations from that point,
using our randomized method described above.

Within the block-coordinate-ascent projection step in the RQPS method,
we set a tolerance of 0.001 for the relative primal-dual gap and
capped the maximum number of iterations of the primal-dual algorithm
to either 10 or 100. We observed that the projection step can be the
most costly part of an RQPS iteration, and an imprecise projection can
often be sufficient to give good performance (though there is
apparently no rigorous theory to guarantee convergence in the presence
of an inexact projection calculation). Ten iterations in the
projection step usually yielded reasonable performance; the algorithm
found a reasonable solution quickly overall, though in some cases the
solution quality fluctuates markedly over time.  With a 100-iteration
limit in the projection subproblem, there is less fluctuation in 2-SUM
value over the course of iterations, but the time per FISTA iteration
increases significantly. (A cap of fewer than 10 yields erratic
convergence behavior, while greater than 100 is too slow.)

In applying Gurobi's interior-point solver to the permutahedron-based
formulation, we set a relative tolerance of 5\% and applied the
randomized procedure described above to recover a permutation from the
final point. We observed that use of a loose tolerance does not
usually degrade the solution quality by more than a percent or two
over results obtained with a tight tolerance, and avoids numerical
instabilities in the interior-point methods.

\begin{figure}[hbtp]
  \begin{center}
    \begin{tabular}{ccc}
      \multicolumn{1}{c}{$n=500$, run 3} \\
      \parbox[c]{52mm}{\includegraphics[width=57mm]{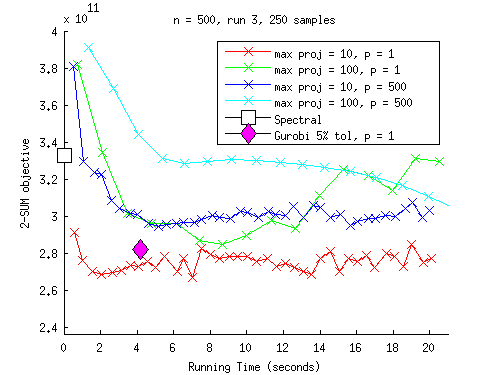}}
      \parbox[c]{52mm}{\includegraphics[width=57mm]{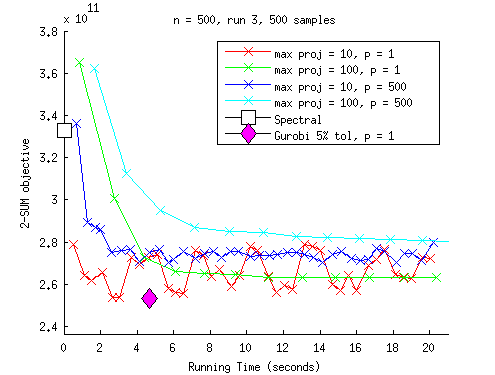}}
      \parbox[c]{55mm}{\includegraphics[width=57mm]{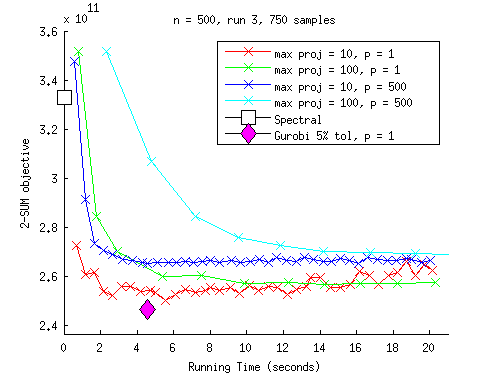}}\\ \\ \\
      \multicolumn{1}{c}{$n=2000$, run 7} \\
      \parbox[c]{52mm}{\includegraphics[width=57mm]{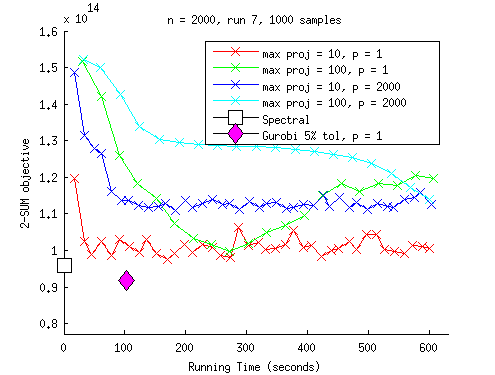}}
      \parbox[c]{52mm}{\includegraphics[width=57mm]{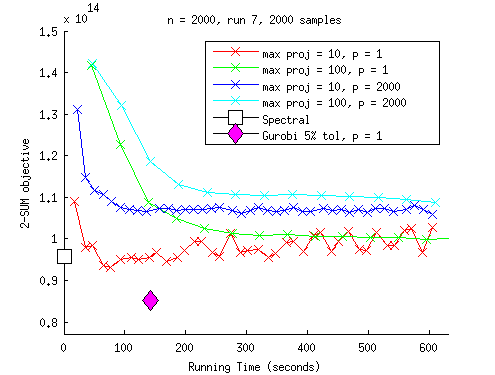}}
      \parbox[c]{55mm}{\includegraphics[width=57mm]{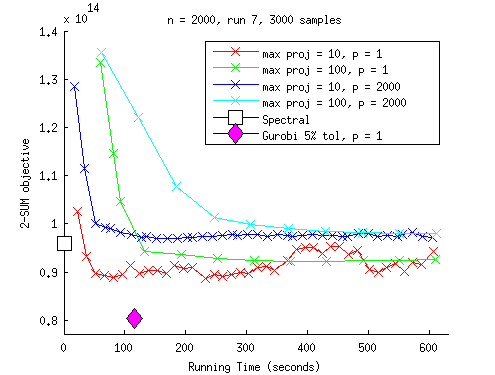}}\\ \\ \\ 
      \multicolumn{1}{c}{$n=5000$, run 8} \\
      \parbox[c]{52mm}{\includegraphics[width=57mm]{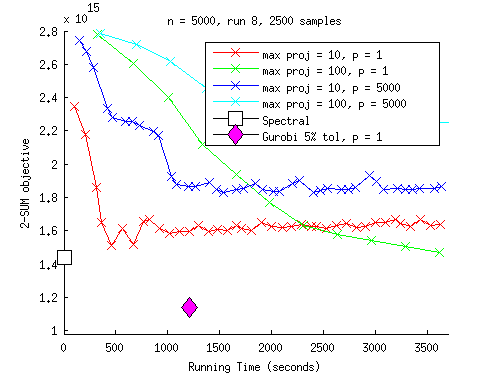}}
      \parbox[c]{52mm}{\includegraphics[width=57mm]{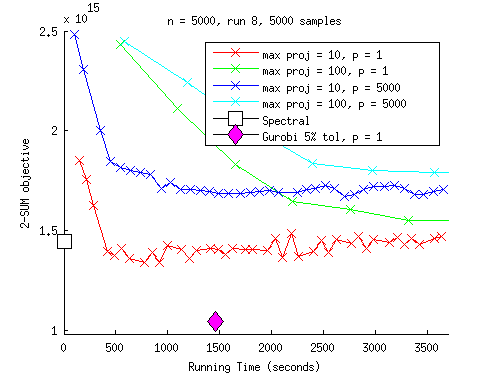}}
      \parbox[c]{55mm}{\includegraphics[width=57mm]{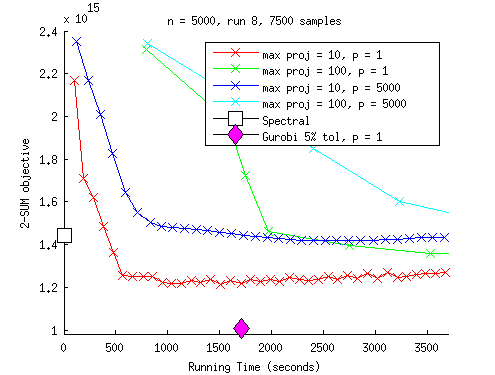}}
    \end{tabular}
  \end{center}
  \caption{Linear Markov Chain Example. Plot of 2-SUM objective over
    time (in seconds) for $n \in \{ 500, 2000, 5000\}$. The red,
    green, blue, and cyan curves represent performance of the RQPS
    code for varying values $p$ and the cap on the maximum number of
    iterations for the projection step. The white square represents
    the spectral solution, and the magenta diamond represents the
    solution returned by Gurobi for the permutahedron-based
    formulation. The horizontal axis in each graph is positioned at
    the 2-SUM objective of the identity permutation on the sample
    covariance matrix.
    \label{fig:markov_rqps}} 
\end{figure}

Figure~\ref{fig:markov_rqps} shows results corresponding to the three
different values of $n$, each row of plots corresponding to $n=500$,
$n=2000$, and $n=5000$.  For each $n$, we choose the run (out of ten)
that shows the best results for RQPS relative to the interior-point
algorithm for the regularized permutahedron-based formulation, and
report remaining runs in Appendix~\ref{sec:additional_experiments}.
Each column of plots corresponds to a different number of side
constraints: $n/2$, $n$, and $3n/2$, as discussed above.  We report
results for four different variants of RQPS, differing according to
the cap on iterations in the projection step and to the value of $p$
in \eqref{eqn:fogel_full}, as follows:
\begin{itemize}
  \item maximum 10 iterations in the projection, with $p=1$;
  \item maximum 100 iterations in the projection, with $p=1$;
  \item maximum 10 iterations in the projection, with $p=n$;
  \item maximum 100 iterations in the projection, with $p=n$.
\end{itemize}
We also report results obtained for the permutahedron-based
formulation and for the spectral formulation (which cannot use side
constraints).

The plots show the quality of solution produced by the various methods
(as measured by the 2-SUM objective on the vertical axis) vs run time
(horizontal axis). For RQPS, with a cap of $10$ iterations within each
projection step, the objective tends to descend to a certain level
rapidly, but then proceeds to fluctuate around that level for the rest
of the running time, or sometimes gets worse.  For a cap of 100
iterations, there is less fluctuation in 2-SUM value, but it takes
some time to produce a solution as good as the previous case. Contrary
to experience reported in \cite{Fogel2013}, values of $p$ greater than
$1$ do not seem to help; our runs for $p=n$ plateaued at higher values
of the 2-SUM objective than those with $p=1$.

We now compare the RQPS results to those obtained with the regularized
permutahedron-based formulation. In most cases, the latter formulation
gives a better solution value, but there are occasional exceptions to
this rule. For example, in the third run for $n = 500$ (the top left
plot in Figure~\ref{fig:markov_rqps}), one variant of RQPS converges
to a solution that is significantly better. Despite its very fast
run times, the spectral method does not yield solutions of competitive
quality, due to not being able to make use of the side constraint
information.

\paragraph{Direct Comparison of Birkhoff and Permutahedron Formulations}

For the second set of experiments, we compare the convergence rate of
the objective value in the Gurobi interior-point solver applied to two
equivalent formulations: the vector-regularized Birkhoff-based convex
formulation \eqref{eqn:fogel_full_vec} with $p=1$ and the regularized
permutahedron-based convex formulation \eqref{eqn:perm_full}. For each
choice of input matrix and sampled ordering information, we first
solved each formulation to within a 0.1\% tolerance (or until reaching
a preset time limit) to obtain a baseline objective $\overline{v}$,
then ran the Gurobi interior-point method for each formulation. At
each iteration, we plot the difference between the primal objective
and $\overline{v}$.

\begin{figure}[hbtp]
  \begin{center}
    \begin{tabular}{ccc}
      \multicolumn{1}{c}{$n=2000$, run 8} \\
      \parbox[c]{52mm}{\includegraphics[width=57mm]{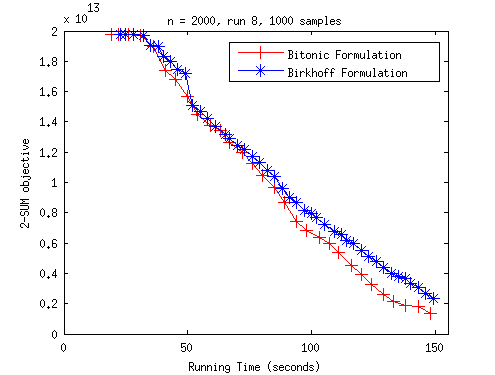}}
      \parbox[c]{52mm}{\includegraphics[width=57mm]{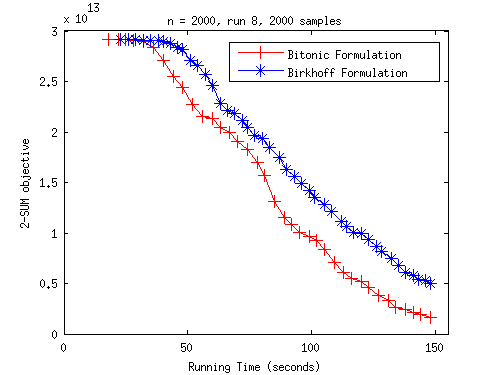}}
      \parbox[c]{55mm}{\includegraphics[width=57mm]{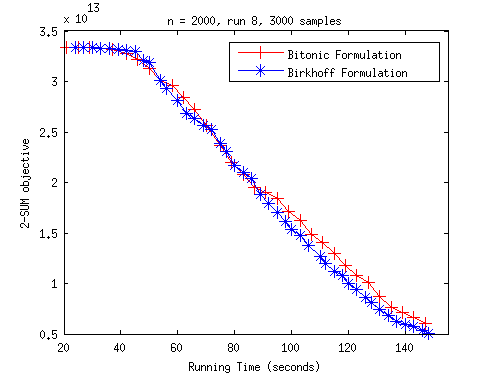}}\\ \\
      \multicolumn{1}{c}{$n=5000$, run 1} \\
      \parbox[c]{52mm}{\includegraphics[width=57mm]{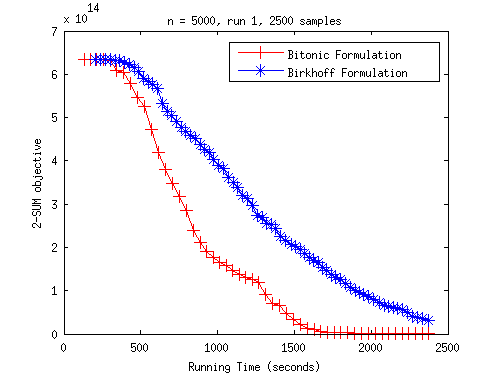}}
      \parbox[c]{52mm}{\includegraphics[width=57mm]{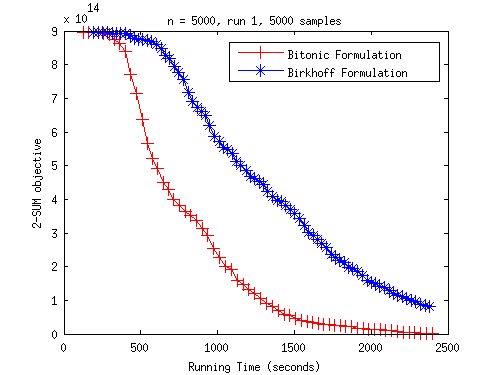}}
      \parbox[c]{55mm}{\includegraphics[width=57mm]{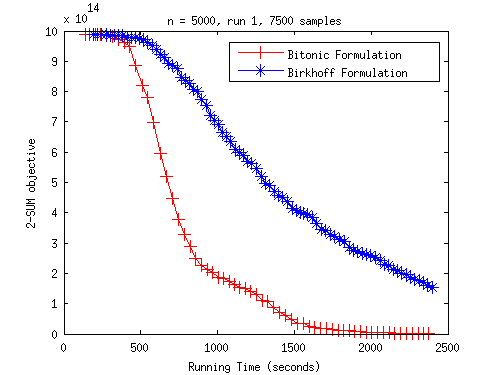}}
    \end{tabular}
  \end{center}
  \caption{Linear Markov Chain Example. Plot of the difference of the
    2-SUM objective from the baseline objective over time (in seconds)
    for $n \in \{ 2000, 5000\}$. The red curve represents performance
    of the permutahedron-based formulation; the blue curve represents
    the performance of the Birkhoff-based formulation. For each value
    of $n$, we display the best run (out of ten) for the
    Birkhoff-based formulation.}
  \label{fig:markov_gurobi}
\end{figure}

Figure~\ref{fig:markov_gurobi} shows the results for the best run (out
of ten) for the Birkhoff-based formulation relative to the
permutahedron-based formulation for $n \in \{2000, 5000\}$.  (Results
for $n = 500$ are omitted because we could not obtain accurate timing
information for short run times.) For $n=2000$, the permutahedron-based
formulation usually converges faster, but for the best run for the
Birkhoff-based formulation they have similar performance.  However,
once we scale up to $n=5000$, the permutahedron-based formulation
converges significantly faster in all tests.

Our comparisons show that the permutahedron-based formulation tends to
yield better solutions in faster times than Birkhoff-based
formulations, regardless of which algorithm is used to solve the
latter. The advantage of the permutahedron-based formulation is more
pronounced when $n$ is large.

\subsection{Additional Empirical Observations} 

We omitted results from the convex formulations when there was no
sampled information, since we have observed that this could lead to
inconsistent results. In general, we have noticed that in many of
those cases the value of the obtained point on the permutahedron lies
in a narrow range, and the randomness inherent in the procedure of
sampling a permutation may be a major factor in determining the
solution quality.

\section{Conclusions and Future Work}
  
In this paper, we bring Goemans' compact description of the
permutahedron into the area of convex optimization, showing that it
can be used to construct a convex relaxation of the 2-SUM problem
problem as introduced by \cite{Fogel2013} --- an optimization in which
the variable is a permutation of $n$ objects.  Enhancements introduced
in the Birkhoff-based formulations can also be applied to the
permutahedron-based formulations. We introduce a new, simpler
regularization scheme that gives solutions of the relaxation that can
be turned into better candidate solutions for the underlying problem.
We also introduce a simple randomized scheme for recovering
permutations from solutions of the relaxation. Empirical results show
that the regularized permutahedron-based formulation gives the best
objective values and converges more rapidly than algorithms applied to
the Birkhoff-based formulation when $n$ is large. Given that the
permutation-based formulation requires only $\Theta(n \log^2 n)$
variables and constraints, whereas the Birkhoff-based formulation
requires $\Theta(n^2)$, the advantage of the permutahedron-based
scheme should continue to grow as $n$ increases.

We hope that this paper spurs further interest in using sorting
networks in the context of other more general classes of permutation
problems, such as graph matching or ranking. A direct adaptation of
this approach is inadequate, since the permutahedron does not uniquely
describe a convex combination of permutations, which is how the
Birkhoff polytope is used in many such problems. However, when the
permutation problem has a solution in the Birkhoff polytope that is
close to an actual permutation, we should expect that the loss of
information when projecting this point in the Birkhoff polytope to the
permutahedron to be insignificant.

\section{Acknowledgements}
We thank Okan Akalin and Taedong Kim for helpful comments and
suggestions for the experiments. We thank the anonymous referees for
feedback that improved the paper's presentation. We also thank the
authors of \cite{Fogel2013} for sharing their experimental code, and
Fajwal Fogel for helpful discussions. 

Lim's work on this project was supported in part by NSF Awards
DMS-0914524 and DMS-1216318, and a grant from ExxonMobil. Wright's
work was supported in part by NSF Award DMS-1216318, ONR Award
N00014-13-1-0129, DOE Award DE-SC0002283, AFOSR Award
FA9550-13-1-0138, and Subcontract 3F-30222 from Argonne National
Laboratory.

{
\bibliographystyle{IEEEtran}
\bibliography{permutahedron}
}


\newpage

\appendix

\section{The Spectral Formulation is a Continuous Relaxation of the 2-SUM Problem}

\label{sec:sf2}

Consider the following problem:
\begin{align} \label{eqn:fiedler_z}
    \min_{x\in\Re^n}\:\:  x^T  L_A  x
    \quad\st\quad  x^T\ones=\frac{n(n+1)}{2}, \;\; \|x\|^2_2=\frac{n(n+1)(2n+1)}{6}.
\end{align} 
The set of permutations lies in the hyperplane defined by
$x^T\ones={n(n+1)}/{2}$, and each permutation has a $\ell_2$-norm of
$\sqrt{n(n+1)(2n+1)/6}$.  Hence, formulation (\ref{eqn:fiedler_z}) is
a continuous relaxation of the 2-SUM problem \eqref{eqn:2sumlap}. If
we shift variables to $z:= x- (({n+1})/{2})\ones$, the objective
becomes
\begin{align*}
    x^T L_A x = z^T L_A z + \left(\frac{n+1}{2}\right)^2 \ones^T L_A \ones = z^T
    L_A z
\end{align*} 
which follows from the fact that the $L_A \ones = 0$.  Note that
$z^T\ones = 0$ if and only if
\begin{align*} 
      x^T\ones = z^T\ones + n\frac{n+1}{2} =
     \frac{n(n+1)}{2},
\end{align*}
and if $z^T\ones = 0$ holds, then $\| z \|^2_2 =
\frac{n(n+1)(2n+1)}{6} - \frac{n(n+1)^2}{4}$ if and only if
\begin{align*}
    \| x \|^2_2 = \left\|z + \frac{n+1}{2}\ones \right\|^2_2
    = \| z \|^2_2 + 2\cdot z^T\ones + \frac{n(n+1)^2}{2} 
    = \| z \|^2_2 + 0 + \frac{n(n+1)^2}{4} = \frac{n(n+1)(2n+1)}{6}.
\end{align*}
This shows that formulations (\ref{eqn:fiedler}) and (\ref{eqn:fiedler_z}) are
equivalent up to constant factors, and hence (\ref{eqn:fiedler}) is also a
continuous relaxation of \eqref{eqn:2sumlap}, up to a constant factor.

\section{Proofs for Sections \ref{sec:permutahedron} and \ref{sec:convex}}

\label{app:proofs}

We provide an alternative proof of Theorem \ref{thm:goemans} and
provide proofs of the results stated when we introduced the convex
relaxations.

\subsection{Alternative Proof of Theorem \ref{thm:goemans}}
Suppose we have a monotonically-increasing vector $v$ and the set $P_v$, the
convex hull of all permutations of $v$. Let $\cSN_v^n$ denote the sorting
network polytope associated with the vector $v$ by replacing the variables on
the output wires by $x^\text{out}_i = v_i$ for all $i \in [n]$.

\begin{proof}
The sorting network ensures that the permutahedron $\cP_v^n$ is contained in
$\{ x^{\text{in}} \:|\; (x^{\text{in}},x^{\text{rest}}) \in \cSN_v^n \}$, and
since the set is convex it also contains the convex hull of the permutahedron.

We will now prove the other containment. Let $y = (y_1,\dotsc,y_n)$ denote the
values on the $n$ wires before a particular comparator $k$, and let $z =
(z_1,\dotsc,z_n)$ denote the values after. It suffices to prove that if $z$ is
in the convex hull of the permutahedron, then so is $y$. Let $a$ and $b$
denote the indices of the wires that are part of comparator $k$. Let $z'$ be
the point obtained by swapping the values of the $a$th and $b$th coordinates.
We have $z \in \cP_v^n$ by assumption, and $z' \in \cP^n_v$ since the
permutahedron is invariant under swapping of coordinates. The constraints
(\ref{eqn:comparator}) ensure that $y$ is a convex combination of points $z$
and $z'$, hence $y \in \cP_v^n$.
\end{proof}

\subsection{Proof of Proposition \ref{prop:perm_norm}}
\begin{proof}
  Every permutation of the vector $v$ has the same $\ell_p$-norm, and
  since the $\ell_p$-norm is strictly-convex for $1 < p < \infty$ and
  $X$ is defined by the convex hull of permutations of $v$, the
  proposition holds.
\end{proof}

\subsection{Proof of Corollary \ref{cor:perm_norm}}
\begin{proof}
  Let $y_1,y_2,\dotsc,y_p$ denote the columns of the matrix $Y$. Since
  \[
    \|P\Pi Y \|_2 = \sqrt{\sum_{i=1}^n \sum_{j=1}^p (P\Pi Y)_{ij}^2} =
    \sqrt{\sum_{j=1}^p \| P\Pi y_j \|^2},
  \]
  it suffices to pick the matrix $\Pi$ to maximize each $\| P\Pi y_j
  \|$ term. Since $P \Pi = \Pi P$, each set $\{ P\Pi y_j \:|\:
  \Pi \in \cB^n \} = \{ \Pi P y_j \:|\: \Pi \in \cB^n \}$ is the
  convex hull of all permutations of $Py_j$. This means that by
  proposition \ref{prop:perm_norm}, the norm $\| P\Pi y_j |$ is
  maximized by choosing $\Pi$ such that $\Pi P y_j$ is a permutation
  of $Py_j$. The entries in $y_j$ are unique so $\Pi$ has to be a
  permutation matrix.
\end{proof}

\section{Characterizing the $\cPH^n \cap \{x\in \bR^n  \:|\: 
x_1 + 1 \leq x_n \}$ polytope} 

\label{sec:tiebreak_characterization}

In this section, we will provide a characterization of the
permutahedron with tiebreaking constraint that will later be useful
in developing first-order algorithms. This characterization will
provide the intuition for why, in the absence of other structured
information, the tiebreaking constraint is the ``best'' single
constraint one can add to the permutahedron for the convex relaxations
we study.

\subsection{Preliminaries}
A simple characterization of the points in the permutahedron of the
permutahedron from the theory of majorization (see
\cite{Marshall2011}) will be useful for our proof.

\begin{lemma}
  A point $x \in \bR^n$ is in $\cPH^n$ if and only if the point $z$
  obtained by sorting the coordinate-wise values of $x$ in descending
  order satisfies the equations
  \begin{align} \label{eqn:z_leq}
    \sum_{i = 1} ^k  z_i \leq \sum_{i=1}^{k} (n + 1 -i)
  \end{align}
  for all $k \in [n]$ and
  \begin{align} \label{eqn:z_eq}
    \sum_{i=1}^n z_i = \frac{n(n+1)}{2}.
  \end{align}
\end{lemma}

\begin{proof}
  Any $x \in \cPH^n$ immediately satisfies the conditions on $z$ since
  the set of equations (\ref{eqn:z_leq}) and (\ref{eqn:z_eq})
  represent a subset of the equations defining the permutahedron. As
  for the other direction, consider $x$ and the corresponding $z$ that
  satisfies equations (\ref{eqn:z_leq}) and (\ref{eqn:z_eq}). Then,
  given $S \subseteq [n]$ we have
  \[
    \sum_{i \in S} x_i \leq \sum_{i = 1}^{|S|} z_i \leq 
      \sum_{i=1}^{|S|} (n + 1 - i).
  \]
  Since this holds for all subsets $S$, $x$ has to be in the
  permutahedron.
\end{proof}

\subsection{Characterization and Proof}
The permutahedron with tiebreaking constraint can be shown to be the
convex hull of all permutations $\pi$ such that $\pi(1) < \pi(n)$,
which is equivalent to the following theorem:
\begin{theorem} \label{thm:tiebreak_polytope}
  Every extreme point of $\cPH^n \cap \{x\in \bR^n  \:|\: x_1 + 1 \leq
  x_n \}$ is a permutation.
\end{theorem}
This implies that the tiebreaking constraint is the single strongest
inequality that one can introduce without cutting off any
permutations. We can optimize linear functions efficiently over this
set \cite{Lawler1978}.


To prove the theorem we only need to study the facet introduced by
adding the tiebreaking constraint $x_1 + 1 \leq x_n$. We will show
that every point on $T = \cPH^n \cap \{x\in \bR^n \:|\: x_1 + 1 = x_n
\}$ that is not a permutation can be expressed as a convex combination
of two other points in the set. In other words, all the extreme points
on the facet $T$ are permutations.

Consider $x \in T$ and let $z$ be the sorted vector such that $z_k$ is
the $k$-th largest element in $x$. Let $\pi$ be a permutation where
$z_{\pi(k)} = x_k$. Since $x_n$ is larger than $x_1$, this means that
$\pi(n) < \pi(1)$. Let $s \in \bR^n$ denote the slack in the
inequalities in equation (\ref{eqn:z_leq}), given by
\[
  s_k = \sum_{i=1}^{k}(n+1-i) -  \sum_{i=1}^{k} z_i
\]
If $x$ is not a permutation, then there must exist indices $a$, $b$
such that $s_a$ is the first non-zero value in $s$ and $s_b$ is the
first zero value after $s_a$. Note that we have
\begin{align}
\label{eqn:z_k} z_k &= n+ 1 -k \text { for } k < a, \\
\label{eqn:z_a} z_a &= n + 1 - a - s_a < n + 1 - a,  \\
\label{eqn:z_b}z_b &= n + 1 - b + s_{b-1} > n + 1 - b. 
\end{align}

We will now show that $x$ can be expressed as a convex combination of
two other points. There are three main cases to consider. The first
(and simplest) case is when $\pi(1) \neq b$ and $\pi(n) \neq a$, and
the second case is when $\pi(n) = a$, and the third case is when
$\pi(1) = b$. (The case where $\pi(1) = a$ or $\pi(n) = b$ does not
occur.)

\begin{lemma} \label{lem:neq_ab}
  $\pi(1) \neq a$ and $\pi(n) \neq b$.
\end{lemma}

\begin{proof}
  Suppose for contradiction that $\pi(1) = a$. Then, there is some
  index $c < a$ such that $\pi(n) = c$, which gives us
  \begin{align*}
    \sum_{i=1}^{a} z_i = \sum_{i=1}^{c} z_i + \sum_{i=c+1}^{a} z_i
    &\geq \left( \sum_{i=1}^{c} (n+1-i)\right) + 
      \sum_{i=c+1}^{a} z_i \\
    &\geq \left( \sum_{i=1}^{c} (n+1-i)\right) + 
      (a- c) \cdot (n + 1 - c - 1) \\
    &\geq \sum_{i=1}^{a} (n+1-i)
  \end{align*}
  where the first inequality follows from (\ref{eqn:z_k}), and the
  second inequality from $z_c \geq z_k\geq z_a = z_c - 1$ for $c < k <
  a$. This is a contradiction because it implies that either the slack
  term $s_a$ is zero or that $x$ is not in the permutahedron.
  
  The argument for $\pi(n) \neq b$ is similar. If $\pi(n) = b$, 
  there is some index $c > b$ such that $\pi(1) = c$. We have 
  \begin{align*}
    \sum_{i=1}^{b+1} z_i = \sum_{i=1}^{b} z_i + z_{b+1}
    &\geq \left( \sum_{i=1}^{b} (n+1-i)\right) + z_{b+1} \\
    &\geq \left( \sum_{i=1}^{b} (n+1-i)\right) + 
      n + 1 - b + s_{b-1} - 1  \\
    &> \sum_{i=1}^{b+1} (n+1-i)
  \end{align*}
  where the second inequality follows from $z_b \geq z_{b+1} \geq z_c
  = z_b - 1$ and (\ref{eqn:z_b}), and the third inequality from the
  fact that the slack $s_{b-1}$ is greater than zero. This is a
  contradiction since it implies $x$ is not in the permutahedron.
\end{proof}

Before we proceed with the case-by-case analysis, we will first
introduce a small $\delta$ term that will help us define two points
such that their convex combination gives us $x$. Consider the terms
$\Delta_k = z_k - z_{k+1}$ and pick $\delta >0$ such that
\[
  \delta < \min \left( \{\Delta_k \:|\: k \in [n-1], \Delta_k > 0
  \} \cup \{ s_i \:|\: a \leq i < b \} \right).
\] 
This choice of $\delta$ is small enough to allow us to take advantage
of the `wiggle-room' that the slack affords us.

We will tackle each of the two cases in separate lemmas.
\begin{lemma} \label{lem:conv_case_1}
  If $\pi(1) \neq b$ and $\pi(n) \neq a$, then the points $v^+$ and
  $v^-$ given by
  \begin{align*}
    v^+_k = 
    \begin{cases} 
      x_k + \delta &\mbox{if } \pi(k) = a \\
      x_k - \delta &\mbox{if } \pi(k) = b \\
      x_k &\mbox{otherwise  }
    \end{cases}
    \qquad\text{and}\qquad 
    v^-_k = 
    \begin{cases} 
      x_k - \delta &\mbox{if }  \pi(k) = a \\
      x_k + \delta &\mbox{if }  \pi(k) = b \\
      x_k &\mbox{otherwise  }
    \end{cases}
  \end{align*}
  satisfy $0.5 v^+ + 0.5 v^- = x$ and are in the permutahedron.
\end{lemma}

\begin{proof}
  Suppose $\pi(1) \neq b$ and $\pi(n) \neq a$. By Lemma
  \ref{lem:neq_ab}, we know that $\pi(1),\pi(n) \notin \{a,b \}$. The
  vectors $v^+$ and $v^-$ exploit this fact by adding or subtracting
  the $\delta$ term from these coordinates.

  To show that $v^+$ is in the polytope, consider the vector $z^+$
  which is the vector $v^+$ sorted in descending order. We will prove
  that for $z^+$ every coordinate $k \in [n]$ satisfies equation
  (\ref{eqn:z_leq}). The choice of $\delta$ ensures that the order
  between elements that are not equal in $z$ is preserved in $z^+$.
  Then, for $k$ such that $a\leq k < b$, we have
  \[ 
    \sum_{i = 1} z^+_i \leq \delta + \sum_{i = 1}^{k} z_i \leq s_k +
    \sum_{i = 1}^k z_i = \sum_{i=1}^{k} (n + 1 -i).  
  \] 
  For $k \geq b$, $\sum_{i = 1}^k z^+_i = \sum_{i = 1}^{k} z_i$, so
  all that remains is to show that equation (\ref{eqn:z_leq}) holds
  for $k < a$. We can prove this by noting that $z^+_k = z_k$ for $k
  < a$. From the slack terms, we know that $z_{a-1} = n - a + 2$ and
  $z_a < n - a + 1 < z_{a-1}$. By the choice of $\delta$, we know that
  $z_a + \delta < z_{a-1}$, hence the ordering of the first $a$ terms
  remains the same in $z^+$ and $z$ and the desired equality holds.
  The proof for $v^-$ is similar.
\end{proof}

The construction of points $v^+$ and $v^-$ for the case where $\pi(n)
= a$ will require a little more care since we can no longer directly
add $\delta$ to $z_a$ and have the point remain in the polytope. 

\begin{lemma} \label{lem:conv_case_2}
  If $\pi(n) = a$, then $x$ is not an extreme point.
\end{lemma}

\begin{proof}
  Let $c = \pi(1)$. An argument similar to the $\pi(n) \neq b$ part of
  Lemma \ref{lem:neq_ab} shows that the indices satisfy $c \leq b$.
  Now we divide the proof into two different cases. For the case where
  $a < c < b$, we construct the points $v^+$ and $v^-$ where
  \begin{align*}
    v^+_k = 
    \begin{cases} 
      x_k + \delta/2 &\mbox{if } k \in \{1,n\}  \\
      x_k - \delta &\mbox{if } \pi(k) = b \\
      x_k &\mbox{otherwise  }
    \end{cases}
    \qquad\text{and}\qquad 
    v^-_k = 
    \begin{cases} 
      x_k - \delta/2 &\mbox{if } k \in \{1,n\}  \\
      x_k + \delta &\mbox{if } \pi(k) = b \\
      x_k &\mbox{otherwise  }
    \end{cases}.
  \end{align*}
  For the $c = b$ case, we have $a < b - 1$ otherwise $a = b - 1$ and
  $z_b = z_a -1$, so $s_a = s_b > 0$ which is not possible due to the
  definition of $b$. Then, we can define $v^+$ and $v^-$ as follows:
  \begin{align*}
    v^+_k = 
    \begin{cases} 
      x_k + \delta/2 &\mbox{if } k \in \{1,n\}  \\
      x_k - \delta &\mbox{if } \pi(k) = b-1 \\
      x_k &\mbox{otherwise  }
    \end{cases}
    \qquad\text{and}\qquad 
    v^-_k = 
    \begin{cases} 
      x_k - \delta/2 &\mbox{if } k \in \{1,n\}  \\
      x_k + \delta &\mbox{if } \pi(k) = b-1 \\
      x_k &\mbox{otherwise  }
    \end{cases}.
  \end{align*}
  The same argument for Lemma \ref{lem:conv_case_1} applies to both of
  these cases.
\end{proof}

The proof for the final case is similar to the second case, with the
roles of $a$ and $b$ swapped. We omit it.

\begin{lemma} \label{lem:conv_case_3}
  If $\pi(1) = b$, then $x$ is not an extreme point.
\end{lemma}

Putting the lemmas together concludes the proof of the theorem. The
result in this section can be slightly generalized. Instead of the
tiebreaking constraint $x_1 + 1 \leq x_n$, one can replace the
constraint with $x_i + k \leq x_j$ for any $k \in [n-1]$ and $i,j \in
[n]$ and prove that every extreme point of the resulting polytope is a 
permutation.

On the other hand, this result only applies to the permutahedron, and
not the convex hull of the permutations of an arbitrary point.  This
proof relies on the fact that the difference between every two
adjacent elements on the sorted permutation vector is a fixed
constant. Given a vector $x \in \bR^n$ ($n > 3$) where this does not
hold and the polytope formed by taking the convex hull of all
permutations of $x$, every non-trivial single inequality that retains
all permutations with $x_1 < x_n$ would create an extreme point that
is not a permutation.

\section{Additional Experiment Details}
\label{sec:additional_experiments}

We include all the results for the Munsingen experiment and the linear
Markov chain experiment.

\begin{table}[htp] 
  \centering
  {
  \begin{tabular}{|c|c|cc|cc|cc|cc|}
    \hline
    Method & $p$ & Reg. Type & Level & 
2-SUM & Std Err & R-score & Std Err & Kendall's $\tau$ & Std Err \\
    \hline
    Input &&&         & 77040 & 0 & 289.0 & 0.0 & 1.000 & 0.000 \\
    Spectral &&&      & 77806 & 0 & 295.0 & 0.0 & 0.755 & 0.001 \\
    \hline
    Permut. & $-$ & none   & $-$  & 72105 & 2908 & 318.0 & 17.7 & 0.891 & 0.015 \\
    Permut. & $-$ & vector & 50\% & 70683 & 2670 & 316.6 & 12.2 & 0.892 & 0.015 \\
    Permut. & $-$ & vector & 90\% & 70075 & 2836 & 311.2 & 13.6 & 0.892 & 0.013 \\
    \hline
    Birkhoff & $n$ & none   & $-$  & 72726 & 3389 & 317.3 & 24.4 & 0.891 & 0.011 \\
    Birkhoff & $n$ & matrix & 50\% & 73198 & 3433 & 314.2 & 14.2 & 0.889 & 0.014 \\
    Birkhoff & $n$ & matrix & 90\% & 72824 & 3334 & 322.1 & 15.3 & 0.889 & 0.016 \\
    Birkhoff & $n$ & vector & 50\% & 71570 & 2878 & 309.3 & 16.5 & 0.890 & 0.015 \\
    Birkhoff & $n$ & vector & 90\% & 70999 & 2964 & 306.4 & 14.4 & 0.892 & 0.010 \\
    \hline
    Birkhoff & $4n$ & none   & $-$  & 73534 & 3556 & 317.0 & 20.0 & 0.889 & 0.011 \\
    Birkhoff & $4n$ & matrix & 50\% & 73177 & 3225 & 321.6 & 21.9 & 0.889 & 0.012 \\
    Birkhoff & $4n$ & matrix & 90\% & 74590 & 4532 & 319.2 & 14.9 & 0.885 & 0.012 \\
    Birkhoff & $4n$ & vector & 50\% & 72440 & 2917 & 317.7 & 17.9 & 0.891 & 0.010 \\
    Birkhoff & $4n$ & vector & 90\% & 71127 & 2805 & 305.8 & 16.6 & 0.891 & 0.011 \\
    \hline
  \end{tabular}
  }
  \caption{Munsingen dataset --- Performance with 38 ordering
  constraints.}
\end{table}

\begin{figure}[hbtp]
  \begin{center}
    \begin{tabular}{ccc}
      \parbox[c]{52mm}{\includegraphics[width=54mm]{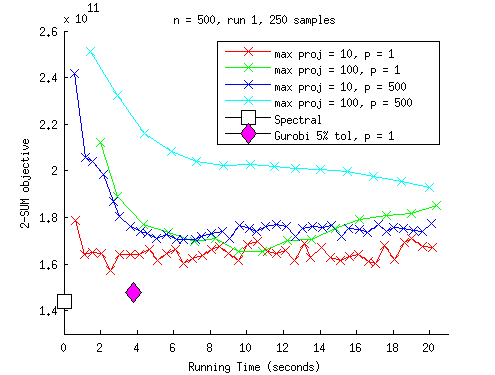}}
      \parbox[c]{52mm}{\includegraphics[width=54mm]{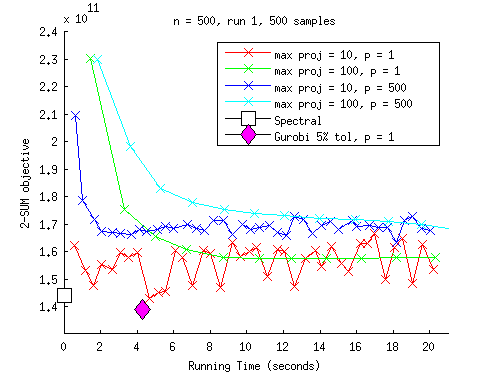}}
      \parbox[c]{55mm}{\includegraphics[width=54mm]{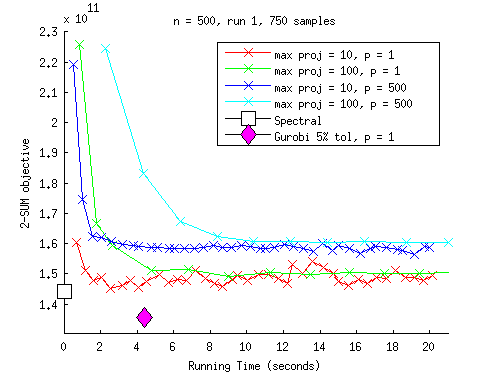}}\\
      
      \parbox[c]{52mm}{\includegraphics[width=54mm]{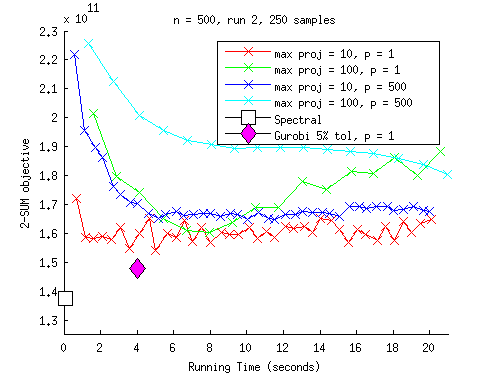}}
      \parbox[c]{52mm}{\includegraphics[width=54mm]{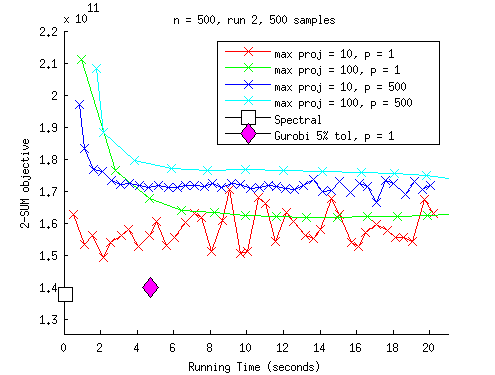}}
      \parbox[c]{55mm}{\includegraphics[width=54mm]{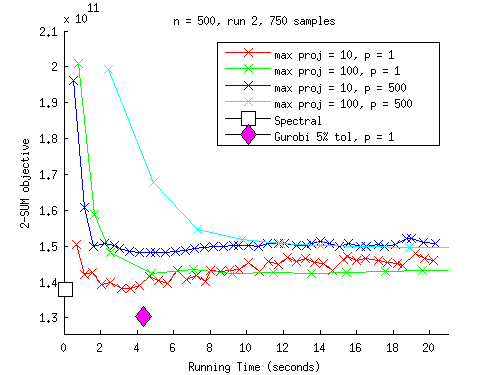}}\\
      
      \parbox[c]{52mm}{\includegraphics[width=54mm]{./fig/mc500_rqps/mc500_run_3_250_samples_.png}}
      \parbox[c]{52mm}{\includegraphics[width=54mm]{./fig/mc500_rqps/mc500_run_3_500_samples_.png}}
      \parbox[c]{55mm}{\includegraphics[width=54mm]{./fig/mc500_rqps/mc500_run_3_750_samples_.png}}\\
      
      \parbox[c]{52mm}{\includegraphics[width=54mm]{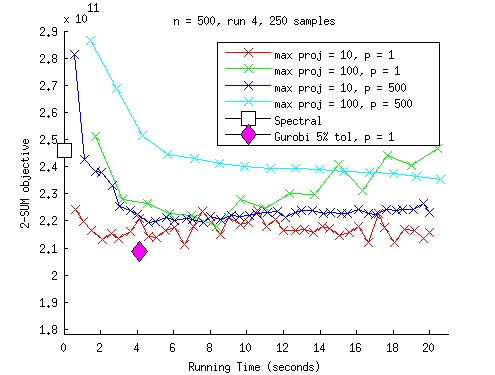}}
      \parbox[c]{52mm}{\includegraphics[width=54mm]{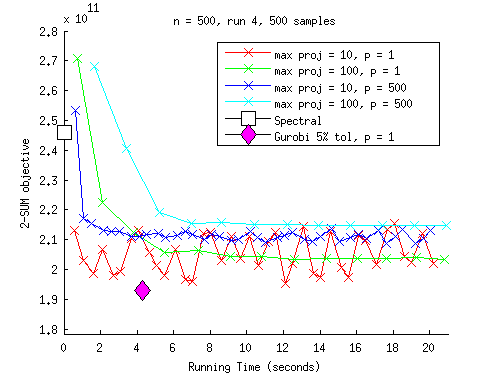}}
      \parbox[c]{55mm}{\includegraphics[width=54mm]{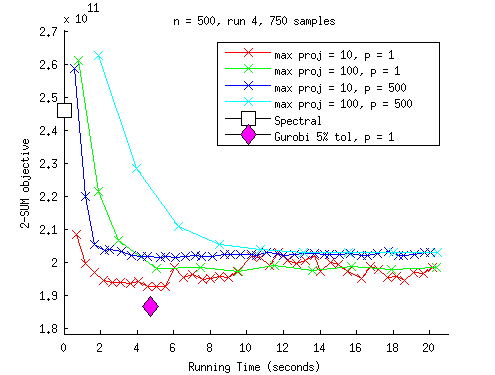}}\\
      
      \parbox[c]{52mm}{\includegraphics[width=54mm]{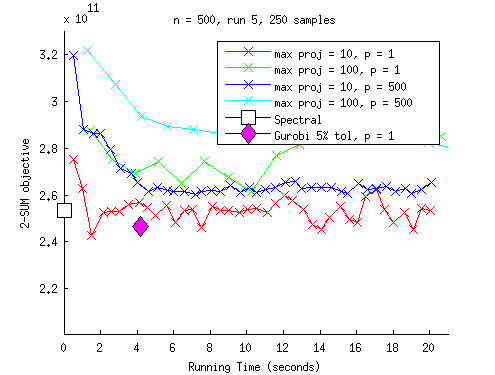}}
      \parbox[c]{52mm}{\includegraphics[width=54mm]{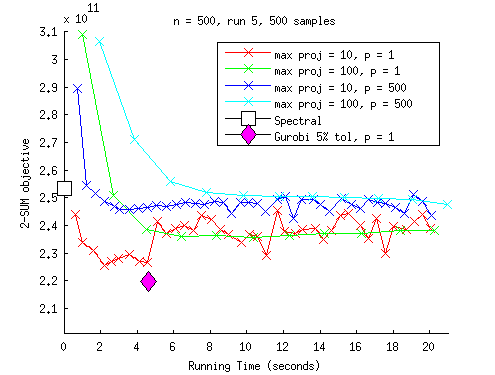}}
      \parbox[c]{55mm}{\includegraphics[width=54mm]{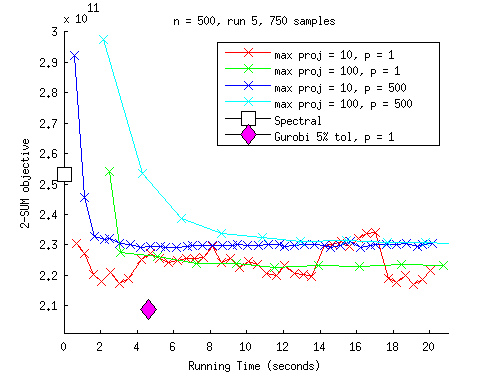}}\\
    \end{tabular}
  \end{center}
  \caption{Linear Markov Chain --- Plot of 2-SUM objective over time
  for $n = 500$ for runs 1 to 5.}
\end{figure}

\begin{figure}[hbtp]
  \begin{center}
    \begin{tabular}{ccc}
      \parbox[c]{52mm}{\includegraphics[width=54mm]{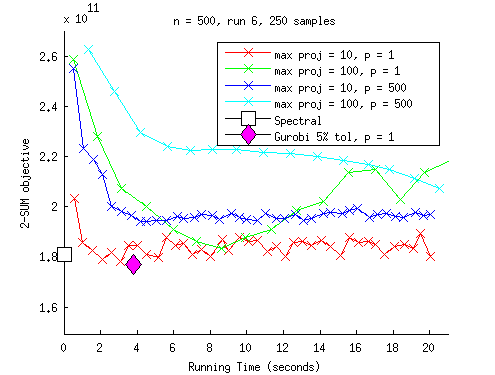}}
      \parbox[c]{52mm}{\includegraphics[width=54mm]{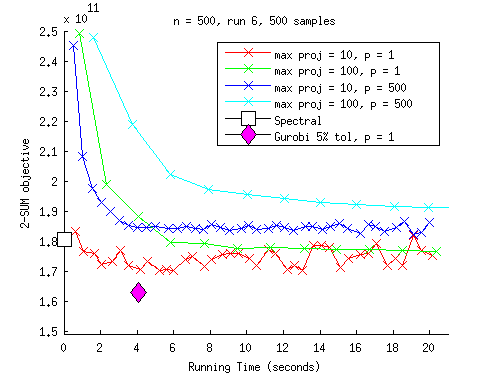}}
      \parbox[c]{55mm}{\includegraphics[width=54mm]{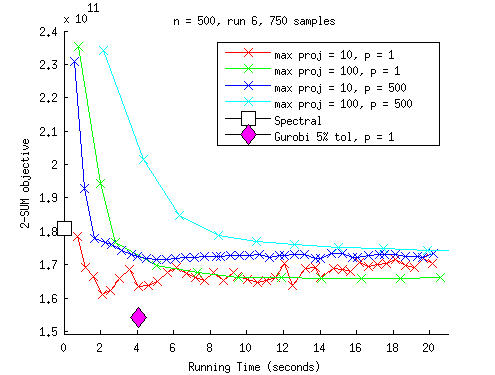}}\\
      
      \parbox[c]{52mm}{\includegraphics[width=54mm]{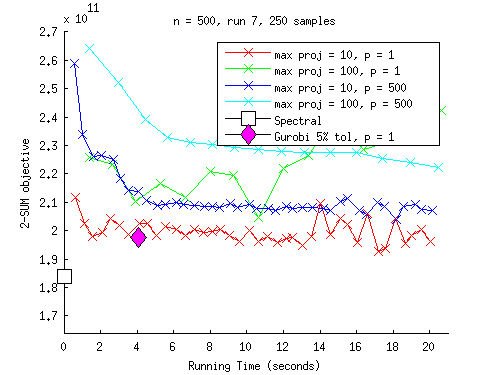}}
      \parbox[c]{52mm}{\includegraphics[width=54mm]{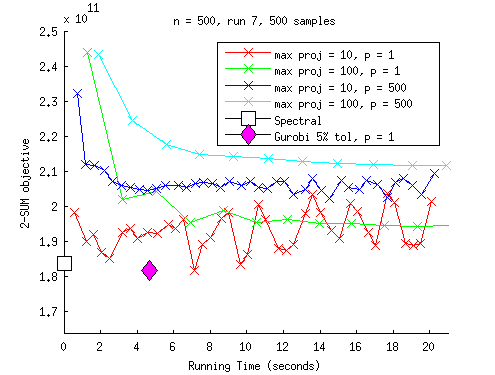}}
      \parbox[c]{55mm}{\includegraphics[width=54mm]{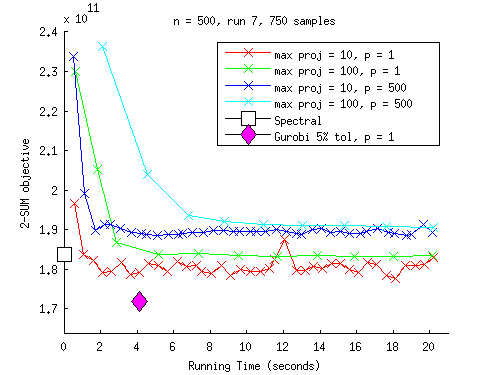}}\\
      
      \parbox[c]{52mm}{\includegraphics[width=54mm]{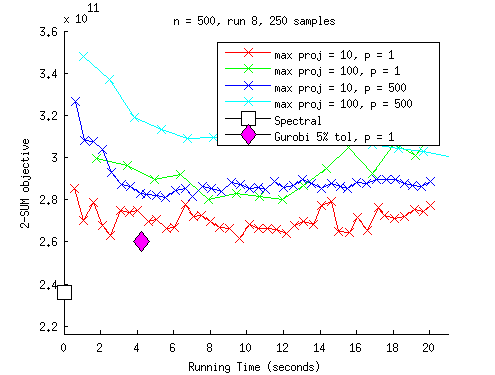}}
      \parbox[c]{52mm}{\includegraphics[width=54mm]{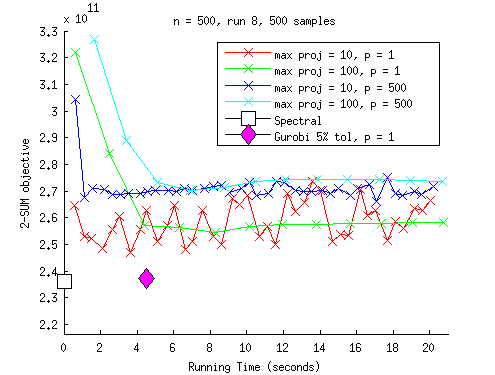}}
      \parbox[c]{55mm}{\includegraphics[width=54mm]{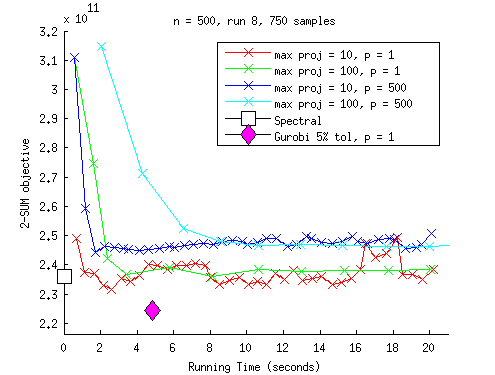}}\\

      \parbox[c]{52mm}{\includegraphics[width=54mm]{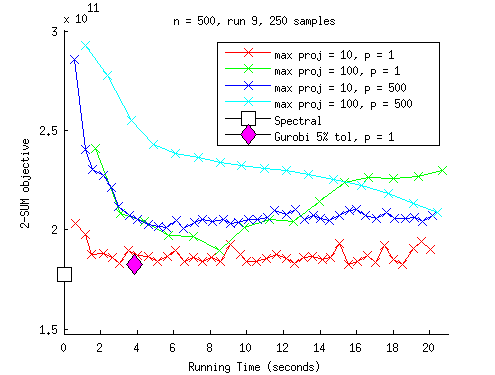}}
      \parbox[c]{52mm}{\includegraphics[width=54mm]{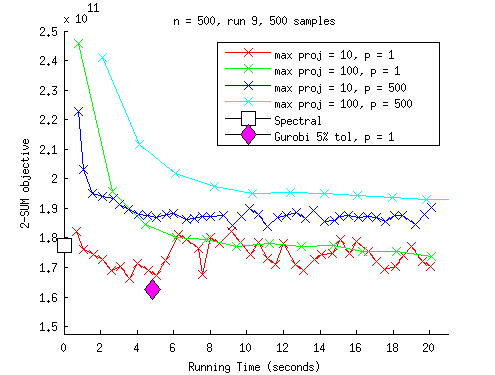}}
      \parbox[c]{55mm}{\includegraphics[width=54mm]{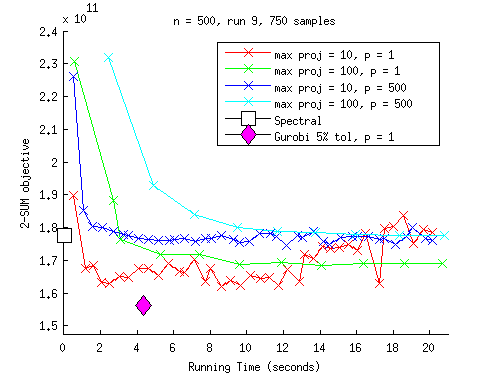}}\\
      
      \parbox[c]{52mm}{\includegraphics[width=54mm]{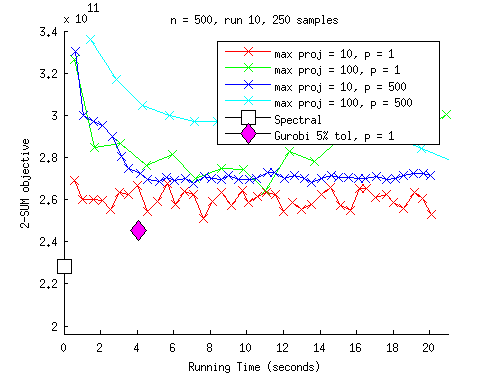}}
      \parbox[c]{52mm}{\includegraphics[width=54mm]{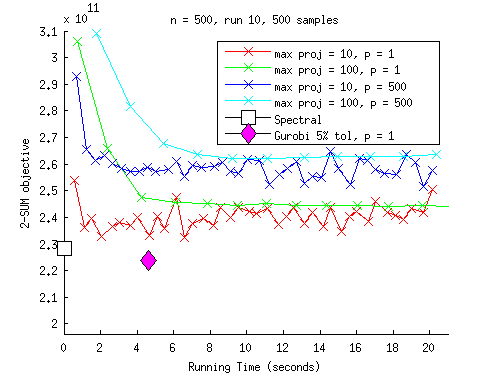}}
      \parbox[c]{55mm}{\includegraphics[width=54mm]{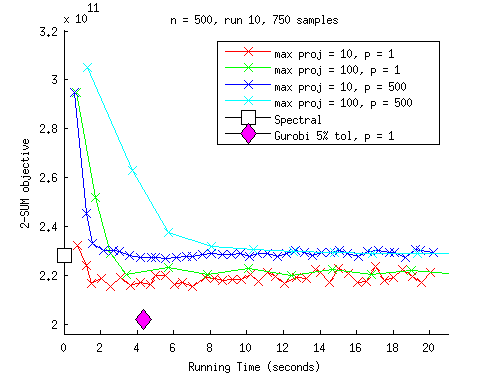}}\\
    \end{tabular}
  \end{center}
  \caption{Linear Markov Chain --- Plot of 2-SUM objective over time
  for $n = 500$ for runs 6 to 10.}
\end{figure}

\begin{figure}[hbtp]
  \begin{center}
    \begin{tabular}{ccc}
      \parbox[c]{52mm}{\includegraphics[width=54mm]{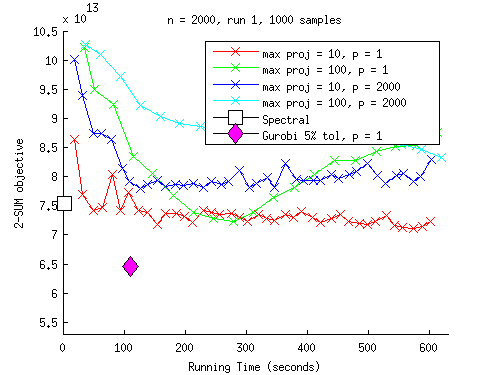}}
      \parbox[c]{52mm}{\includegraphics[width=54mm]{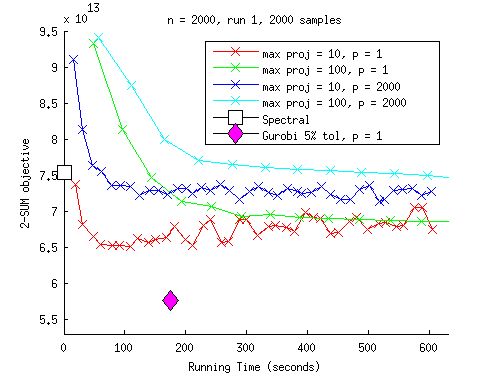}}
      \parbox[c]{55mm}{\includegraphics[width=54mm]{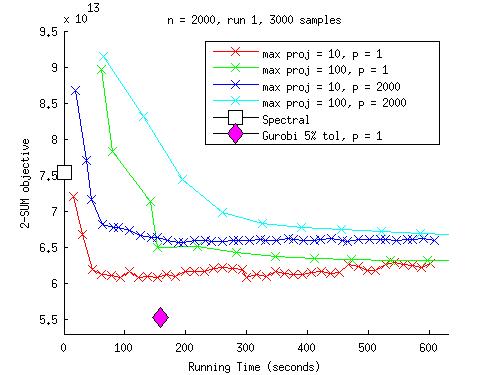}}\\
      
      \parbox[c]{52mm}{\includegraphics[width=54mm]{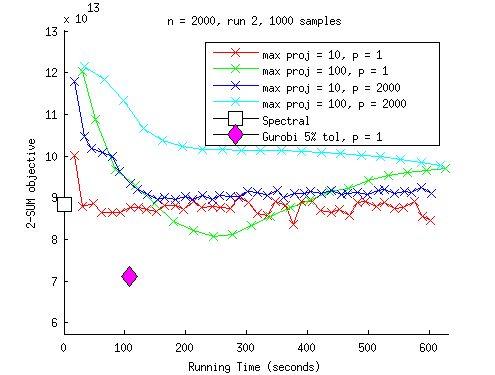}}
      \parbox[c]{52mm}{\includegraphics[width=54mm]{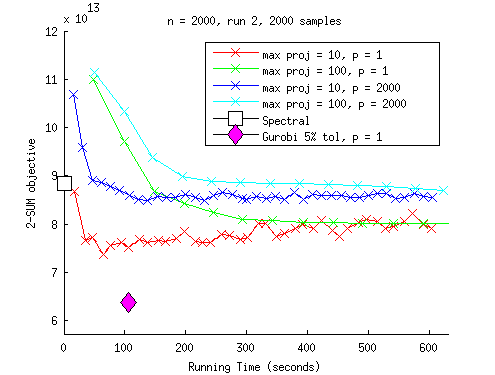}}
      \parbox[c]{55mm}{\includegraphics[width=54mm]{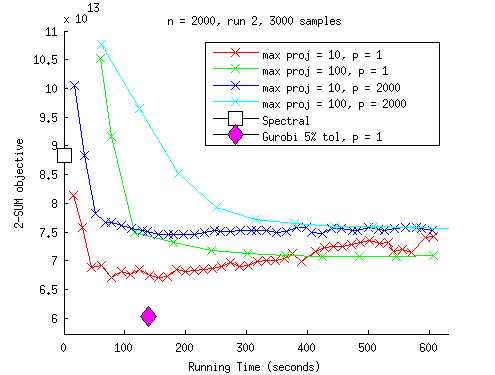}}\\
      
      \parbox[c]{52mm}{\includegraphics[width=54mm]{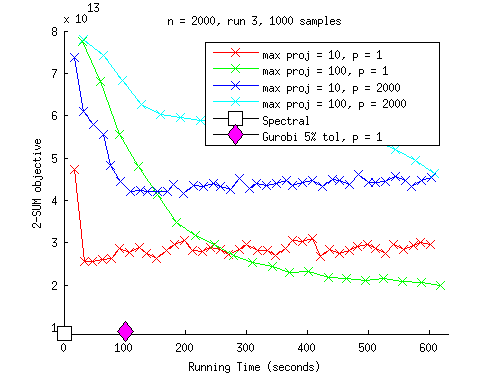}}
      \parbox[c]{52mm}{\includegraphics[width=54mm]{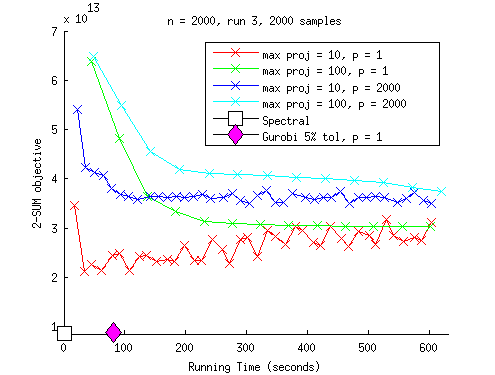}}
      \parbox[c]{55mm}{\includegraphics[width=54mm]{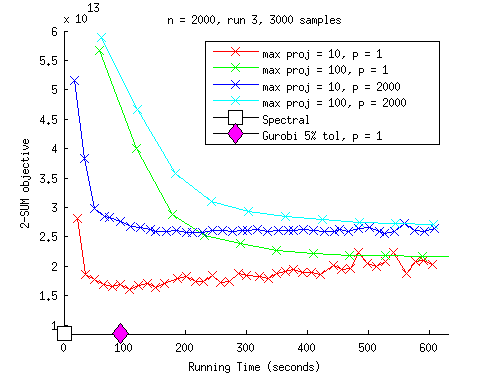}}\\
      
      \parbox[c]{52mm}{\includegraphics[width=54mm]{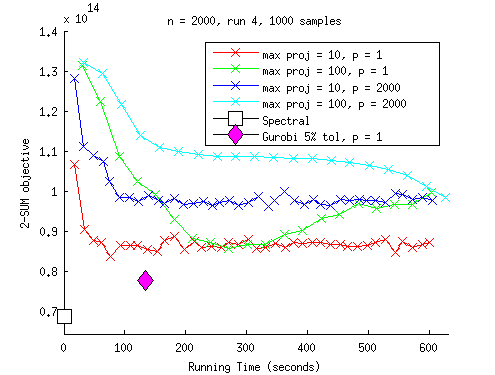}}
      \parbox[c]{52mm}{\includegraphics[width=54mm]{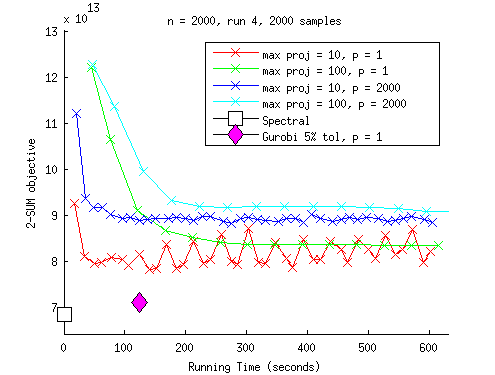}}
      \parbox[c]{55mm}{\includegraphics[width=54mm]{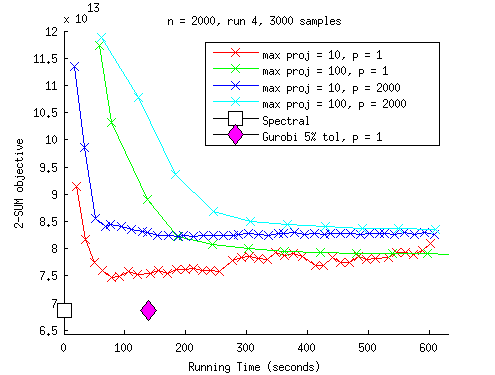}}\\
      
      \parbox[c]{52mm}{\includegraphics[width=54mm]{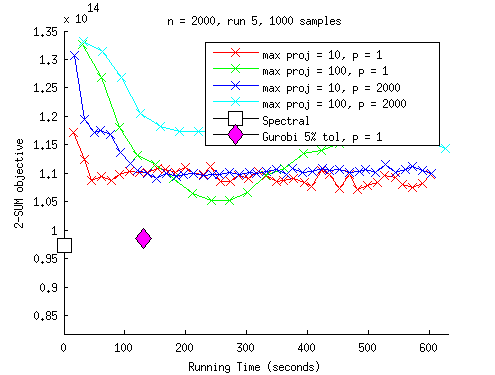}}
      \parbox[c]{52mm}{\includegraphics[width=54mm]{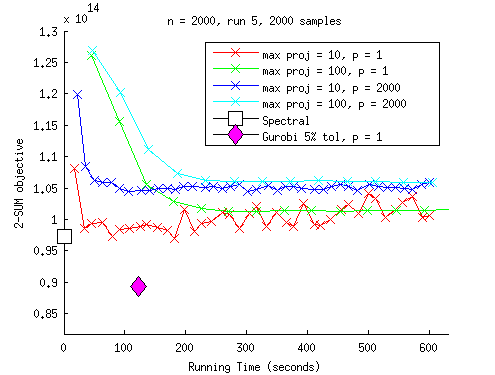}}
      \parbox[c]{55mm}{\includegraphics[width=54mm]{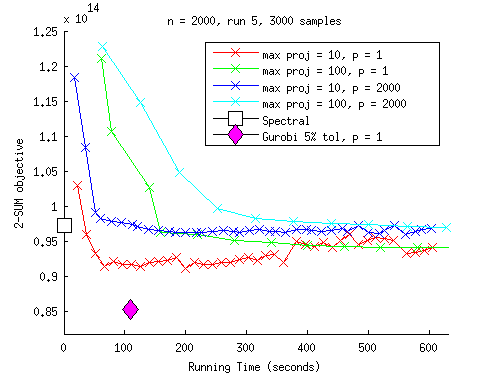}}\\
    \end{tabular}
  \end{center}
  \caption{Linear Markov Chain --- Plot of 2-SUM objective over time
  for $n = 2000$ for runs 1 to 5.}
\end{figure}

\begin{figure}[hbtp]
  \begin{center}
    \begin{tabular}{ccc}
      \parbox[c]{52mm}{\includegraphics[width=54mm]{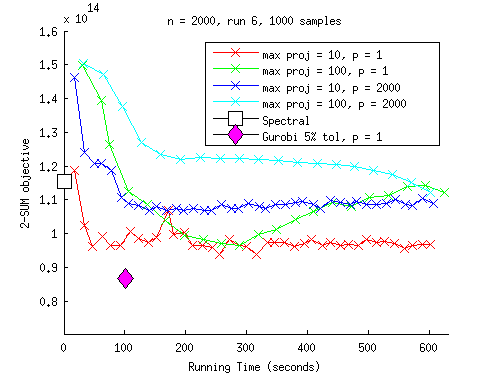}}
      \parbox[c]{52mm}{\includegraphics[width=54mm]{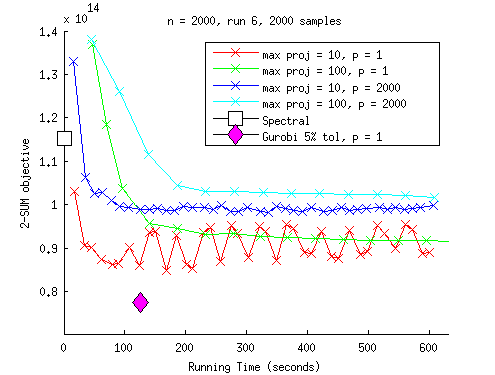}}
      \parbox[c]{55mm}{\includegraphics[width=54mm]{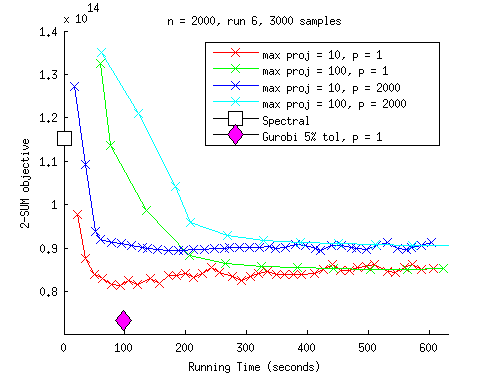}}\\
      
      \parbox[c]{52mm}{\includegraphics[width=54mm]{./fig/mc2000_rqps/mc2000_run_7_1000_samples_.png}}
      \parbox[c]{52mm}{\includegraphics[width=54mm]{./fig/mc2000_rqps/mc2000_run_7_2000_samples_.png}}
      \parbox[c]{55mm}{\includegraphics[width=54mm]{./fig/mc2000_rqps/mc2000_run_7_3000_samples_.png}}\\
      
      \parbox[c]{52mm}{\includegraphics[width=54mm]{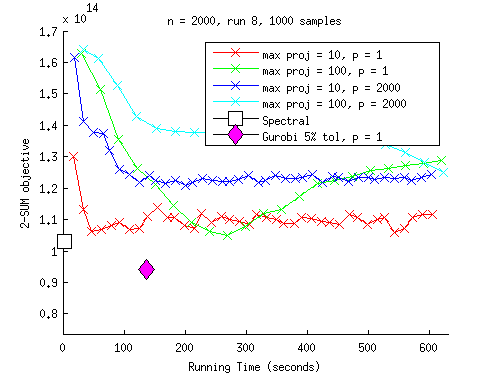}}
      \parbox[c]{52mm}{\includegraphics[width=54mm]{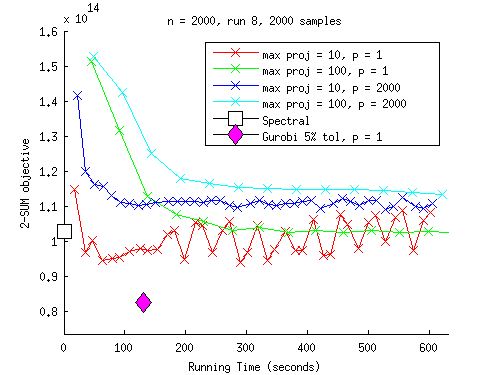}}
      \parbox[c]{55mm}{\includegraphics[width=54mm]{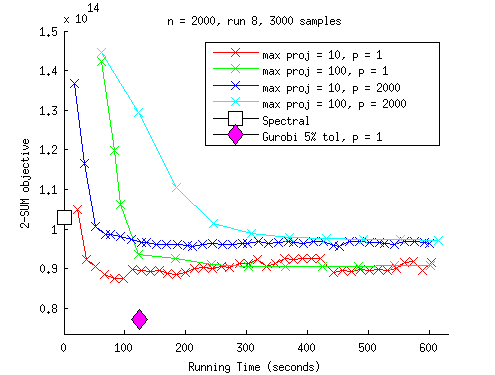}}\\

      \parbox[c]{52mm}{\includegraphics[width=54mm]{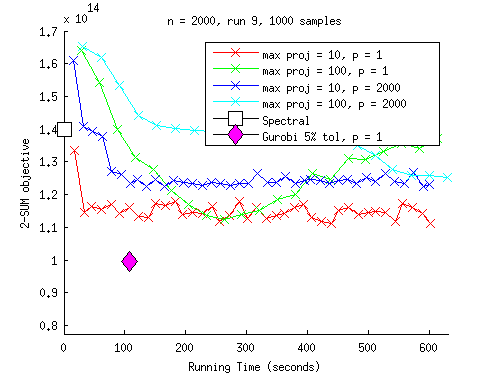}}
      \parbox[c]{52mm}{\includegraphics[width=54mm]{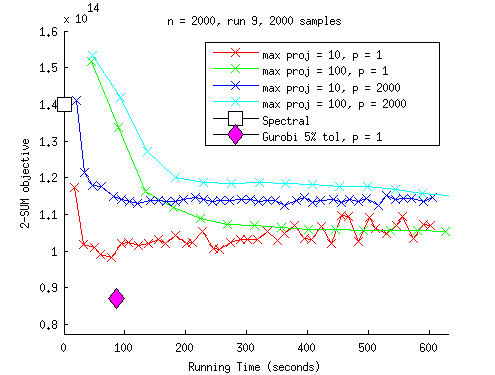}}
      \parbox[c]{55mm}{\includegraphics[width=54mm]{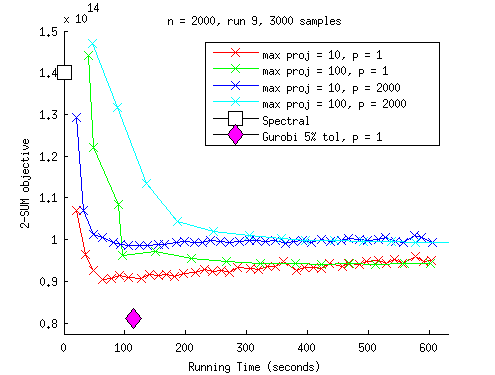}}\\
      
      \parbox[c]{52mm}{\includegraphics[width=54mm]{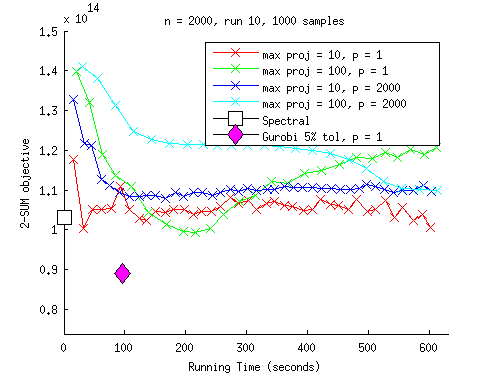}}
      \parbox[c]{52mm}{\includegraphics[width=54mm]{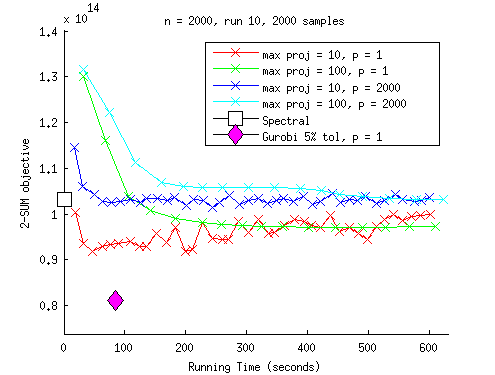}}
      \parbox[c]{55mm}{\includegraphics[width=54mm]{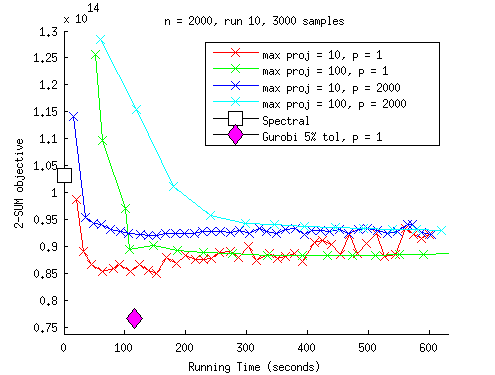}}\\
    \end{tabular}
  \end{center}
  \caption{Linear Markov Chain --- Plot of 2-SUM objective over time
  for $n = 2000$ for runs 6 to 10.}
\end{figure}

\begin{figure}[hbtp]
  \begin{center}
    \begin{tabular}{ccc}
      \parbox[c]{52mm}{\includegraphics[width=54mm]{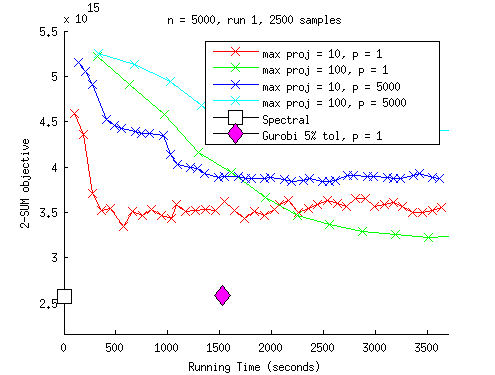}}
      \parbox[c]{52mm}{\includegraphics[width=54mm]{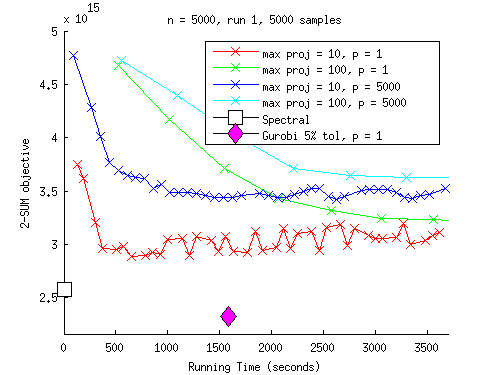}}
      \parbox[c]{55mm}{\includegraphics[width=54mm]{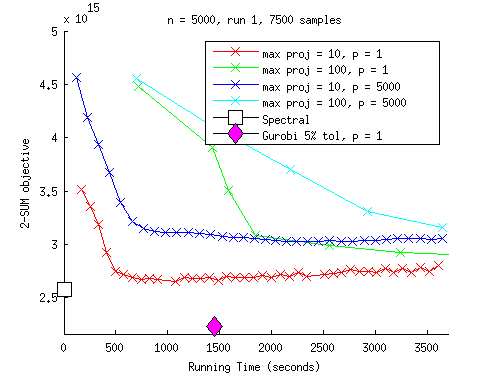}}\\
      
      \parbox[c]{52mm}{\includegraphics[width=54mm]{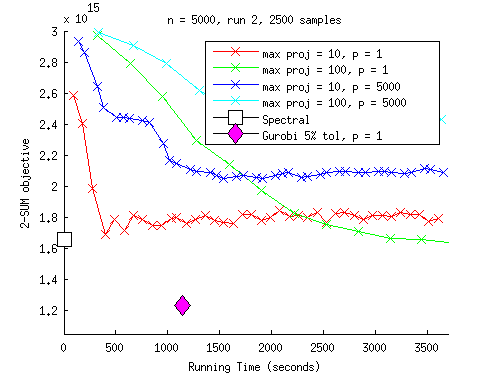}}
      \parbox[c]{52mm}{\includegraphics[width=54mm]{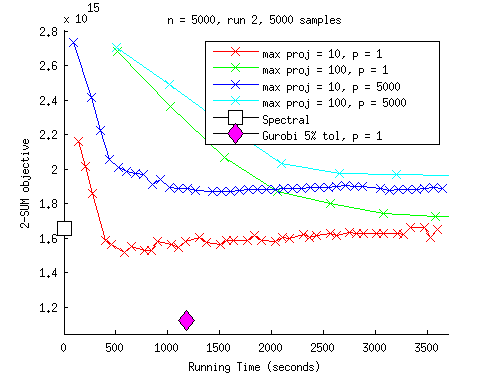}}
      \parbox[c]{55mm}{\includegraphics[width=54mm]{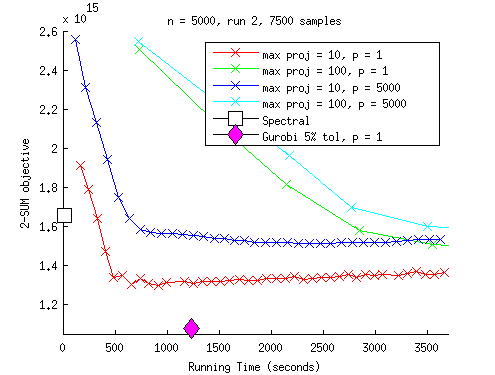}}\\
      
      \parbox[c]{52mm}{\includegraphics[width=54mm]{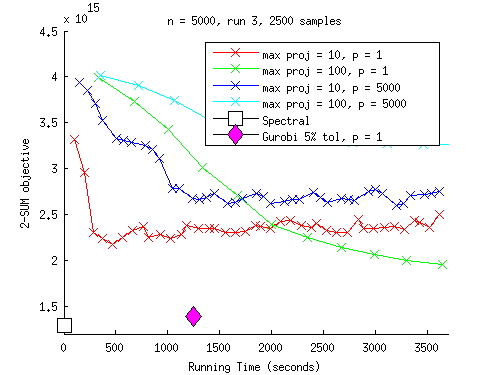}}
      \parbox[c]{52mm}{\includegraphics[width=54mm]{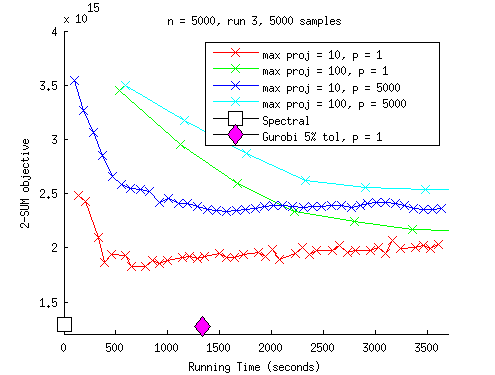}}
      \parbox[c]{55mm}{\includegraphics[width=54mm]{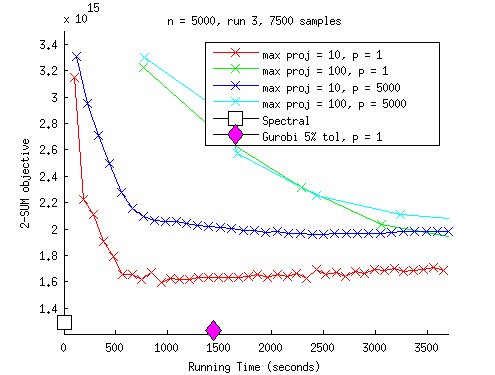}}\\
      
      \parbox[c]{52mm}{\includegraphics[width=54mm]{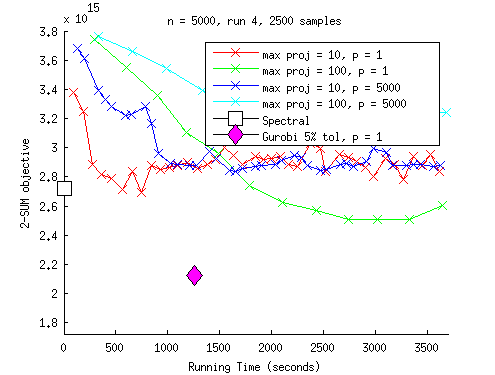}}
      \parbox[c]{52mm}{\includegraphics[width=54mm]{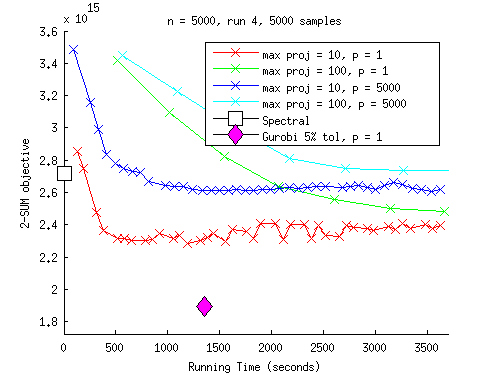}}
      \parbox[c]{55mm}{\includegraphics[width=54mm]{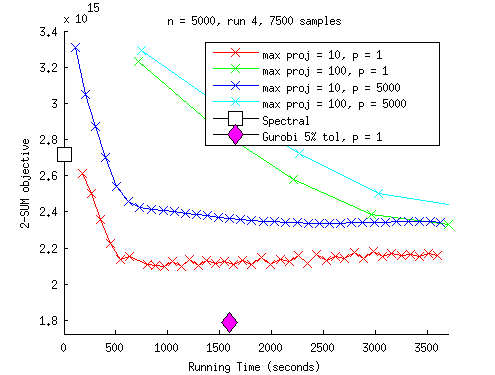}}\\
      
      \parbox[c]{52mm}{\includegraphics[width=54mm]{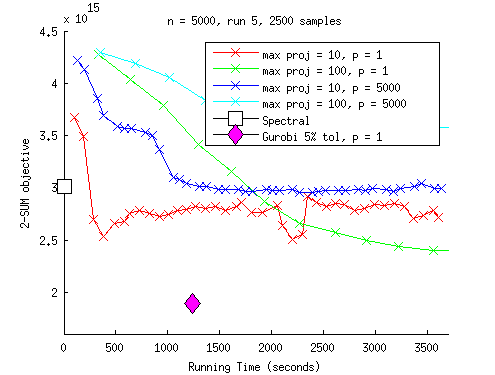}}
      \parbox[c]{52mm}{\includegraphics[width=54mm]{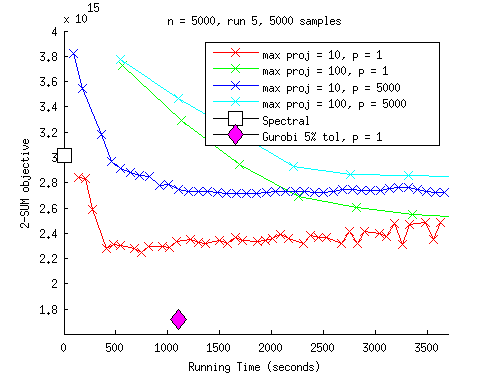}}
      \parbox[c]{55mm}{\includegraphics[width=54mm]{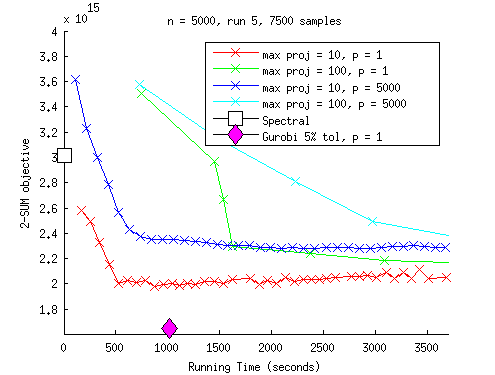}}\\
    \end{tabular}
  \end{center}
  \caption{Linear Markov Chain --- Plot of 2-SUM objective over time
  for $n = 5000$ for runs 1 to 5.}
\end{figure}

\begin{figure}[hbtp]
  \begin{center}
    \begin{tabular}{ccc}
      \parbox[c]{52mm}{\includegraphics[width=54mm]{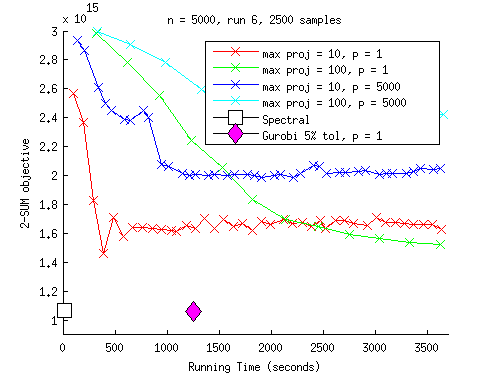}}
      \parbox[c]{52mm}{\includegraphics[width=54mm]{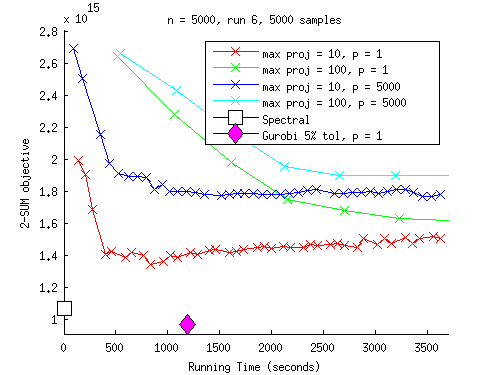}}
      \parbox[c]{55mm}{\includegraphics[width=54mm]{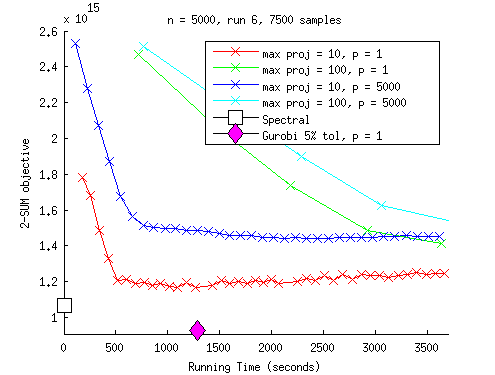}}\\
      
      \parbox[c]{52mm}{\includegraphics[width=54mm]{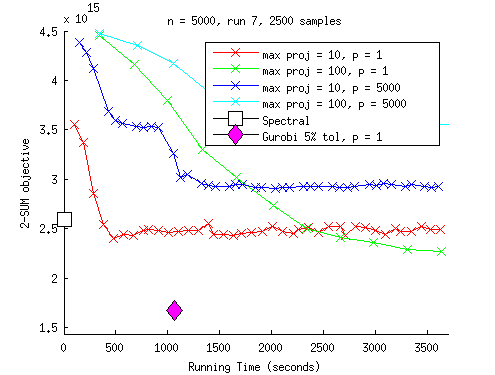}}
      \parbox[c]{52mm}{\includegraphics[width=54mm]{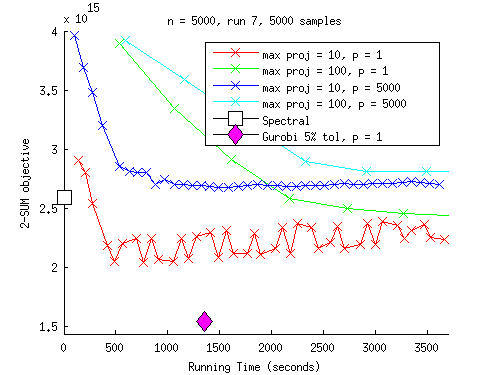}}
      \parbox[c]{55mm}{\includegraphics[width=54mm]{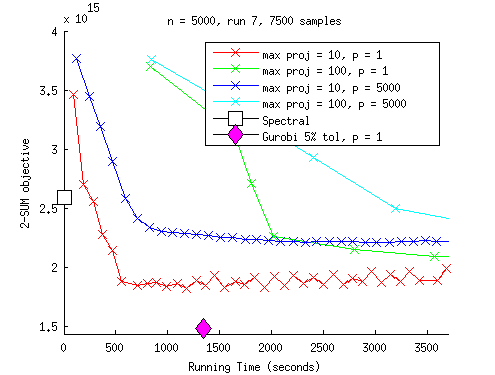}}\\
      
      \parbox[c]{52mm}{\includegraphics[width=54mm]{./fig/mc5000_rqps/mc5000_run_8_2500_samples_.png}}
      \parbox[c]{52mm}{\includegraphics[width=54mm]{./fig/mc5000_rqps/mc5000_run_8_5000_samples_.png}}
      \parbox[c]{55mm}{\includegraphics[width=54mm]{./fig/mc5000_rqps/mc5000_run_8_7500_samples_.png}}\\

      \parbox[c]{52mm}{\includegraphics[width=54mm]{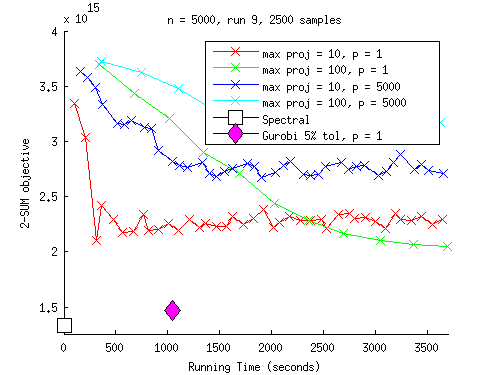}}
      \parbox[c]{52mm}{\includegraphics[width=54mm]{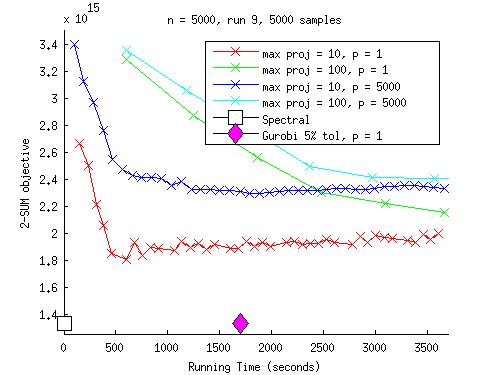}}
      \parbox[c]{55mm}{\includegraphics[width=54mm]{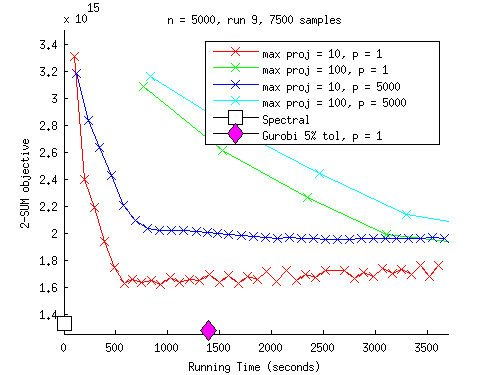}}\\
      
      \parbox[c]{52mm}{\includegraphics[width=54mm]{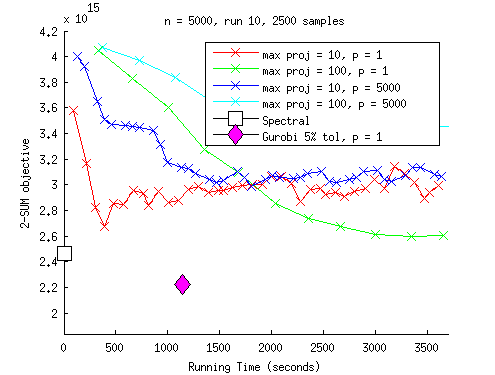}}
      \parbox[c]{52mm}{\includegraphics[width=54mm]{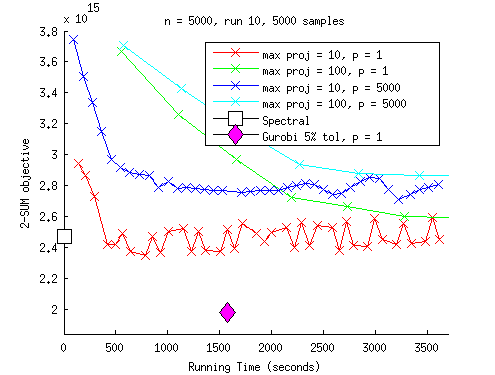}}
      \parbox[c]{55mm}{\includegraphics[width=54mm]{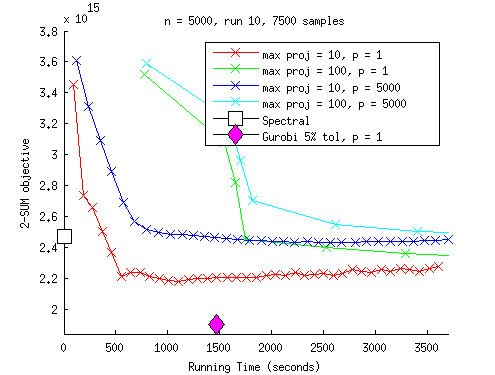}}\\
    \end{tabular}
  \end{center}
  \caption{Linear Markov Chain --- Plot of 2-SUM objective over time
  for $n = 5000$ for runs 6 to 10.}
\end{figure}

\begin{figure}[hbtp]
  \begin{center}
    \begin{tabular}{ccc}
      \parbox[c]{52mm}{\includegraphics[width=54mm]{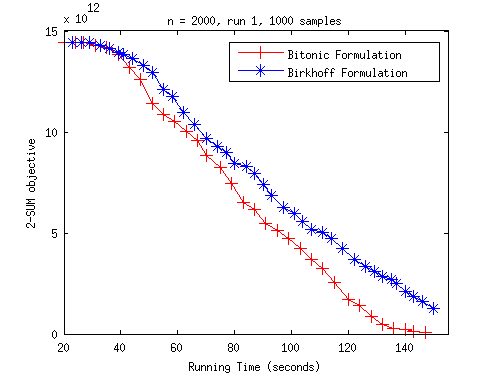}}
      \parbox[c]{52mm}{\includegraphics[width=54mm]{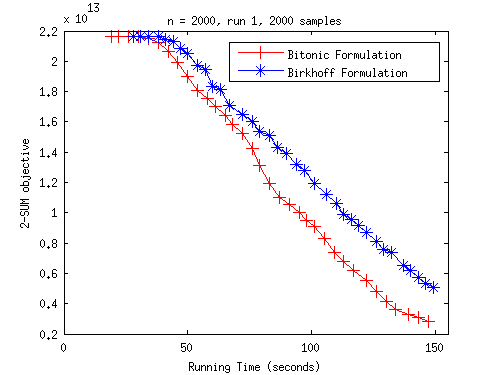}}
      \parbox[c]{55mm}{\includegraphics[width=54mm]{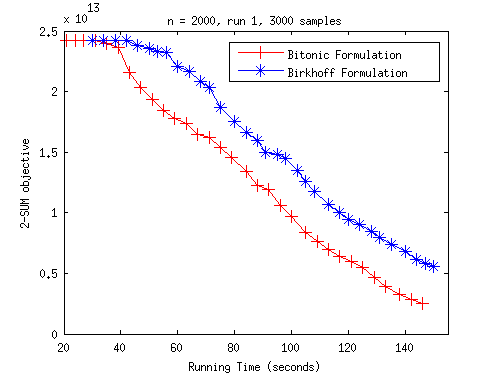}}\\

      \parbox[c]{52mm}{\includegraphics[width=54mm]{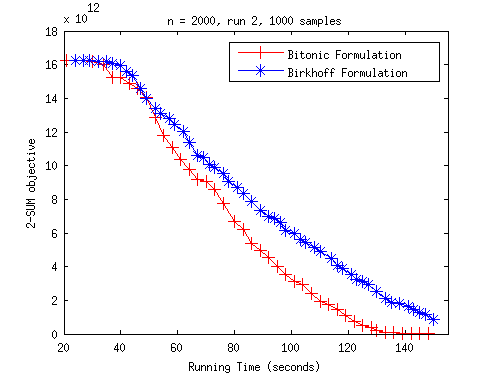}}
      \parbox[c]{52mm}{\includegraphics[width=54mm]{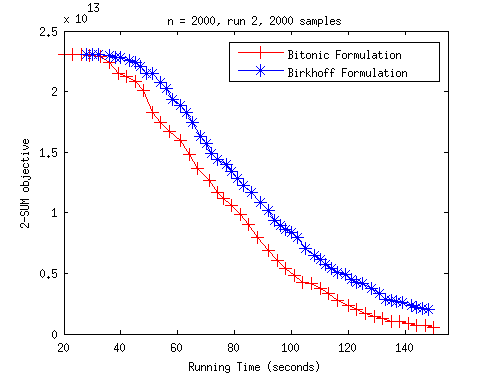}}
      \parbox[c]{55mm}{\includegraphics[width=54mm]{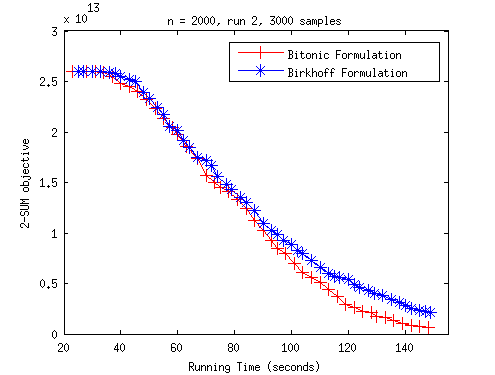}}\\

      \parbox[c]{52mm}{\includegraphics[width=54mm]{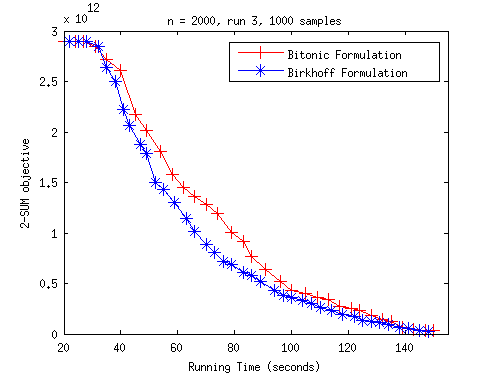}}
      \parbox[c]{52mm}{\includegraphics[width=54mm]{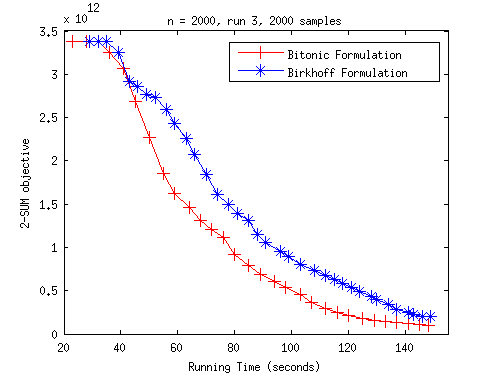}}
      \parbox[c]{55mm}{\includegraphics[width=54mm]{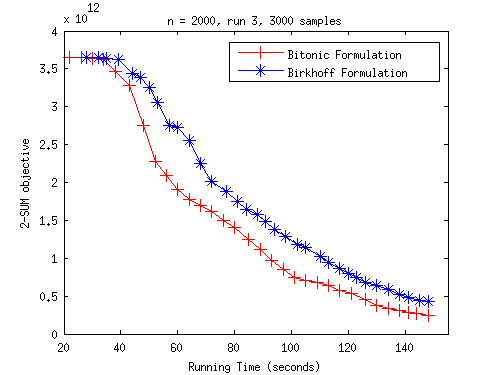}}\\

      \parbox[c]{52mm}{\includegraphics[width=54mm]{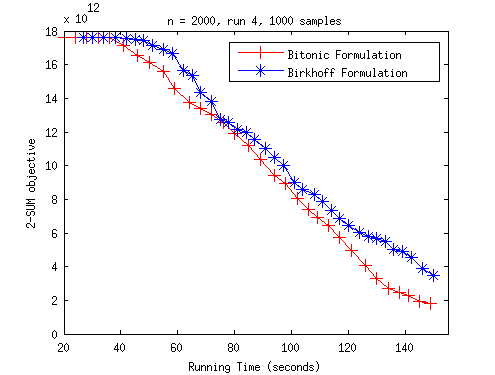}}
      \parbox[c]{52mm}{\includegraphics[width=54mm]{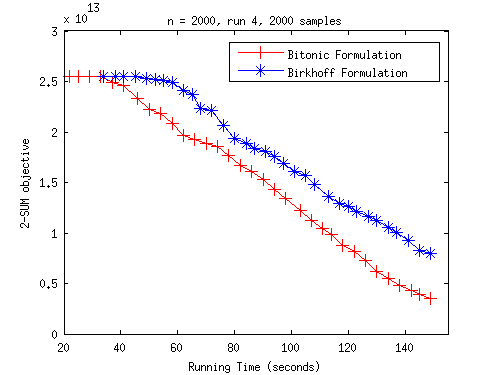}}
      \parbox[c]{55mm}{\includegraphics[width=54mm]{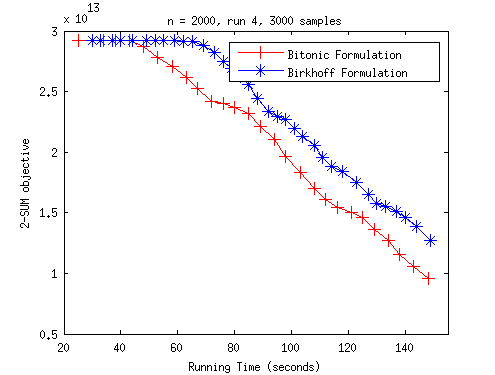}}\\ 
      
      \parbox[c]{52mm}{\includegraphics[width=54mm]{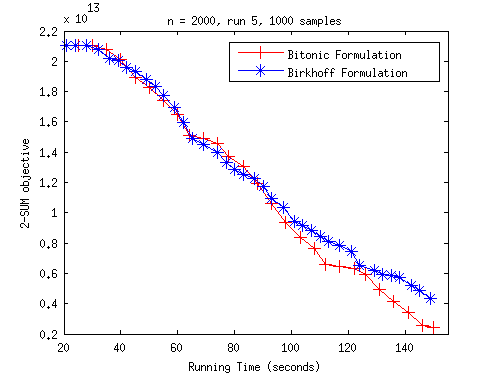}}
      \parbox[c]{52mm}{\includegraphics[width=54mm]{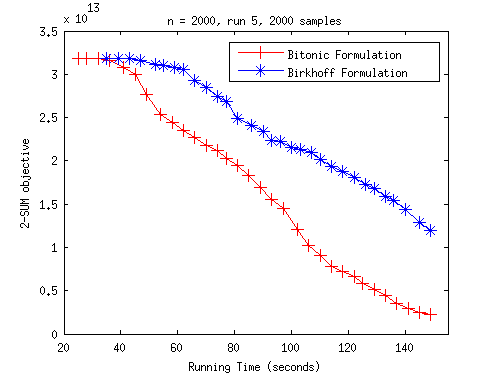}}
      \parbox[c]{55mm}{\includegraphics[width=54mm]{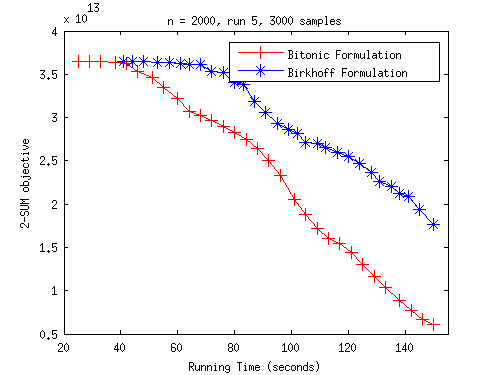}}\\ 
    \end{tabular}
  \end{center}
  \caption{Linear Markov Chain --- Plot of the difference of the 2-SUM
    objective from the baseline objective over time for $n = 2000$ for
    runs 1 to 5.}
\end{figure}

\begin{figure}[hbtp]
  \begin{center}
    \begin{tabular}{ccc}
      \parbox[c]{52mm}{\includegraphics[width=54mm]{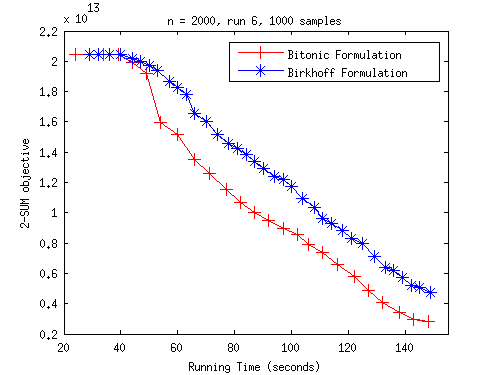}}
      \parbox[c]{52mm}{\includegraphics[width=54mm]{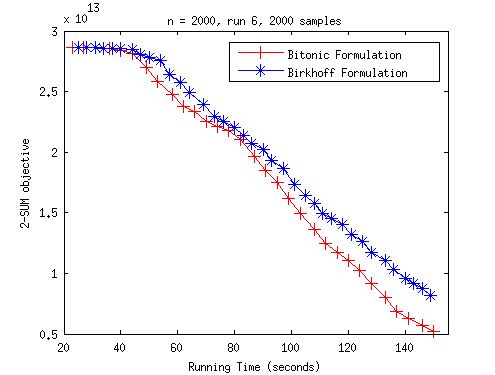}}
      \parbox[c]{55mm}{\includegraphics[width=54mm]{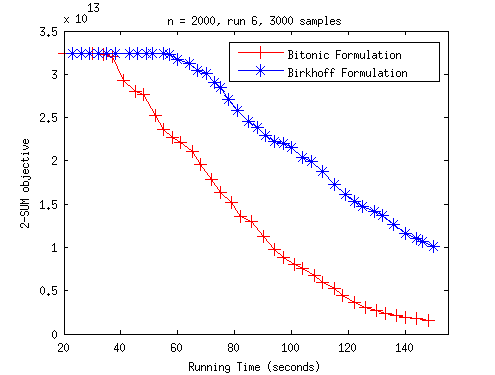}}\\ 
      
      \parbox[c]{52mm}{\includegraphics[width=54mm]{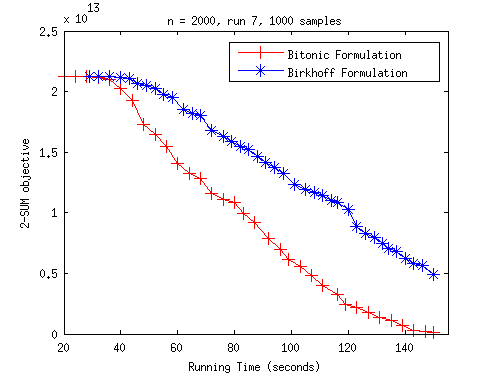}}
      \parbox[c]{52mm}{\includegraphics[width=54mm]{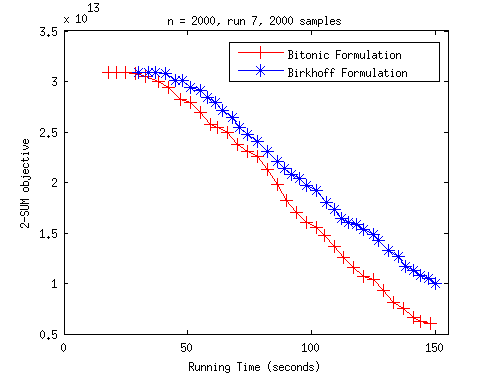}}
      \parbox[c]{55mm}{\includegraphics[width=54mm]{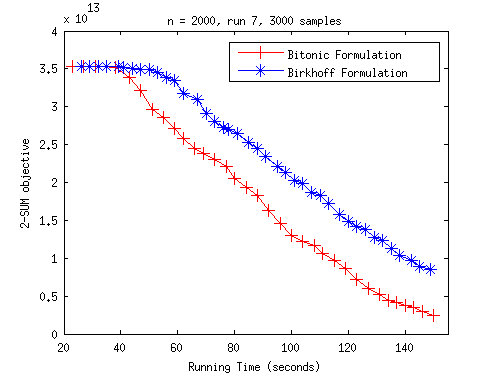}}\\ 
      
      \parbox[c]{52mm}{\includegraphics[width=54mm]{./fig/mc2000_grb/mc2000_grb_run_8_1000_samples_.png}}
      \parbox[c]{52mm}{\includegraphics[width=54mm]{./fig/mc2000_grb/mc2000_grb_run_8_2000_samples_.png}}
      \parbox[c]{55mm}{\includegraphics[width=54mm]{./fig/mc2000_grb/mc2000_grb_run_8_3000_samples_.png}}\\ 
      
      \parbox[c]{52mm}{\includegraphics[width=54mm]{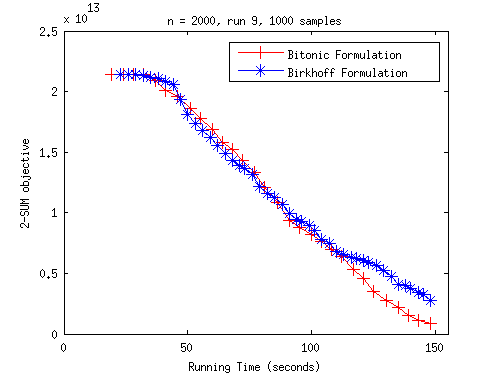}}
      \parbox[c]{52mm}{\includegraphics[width=54mm]{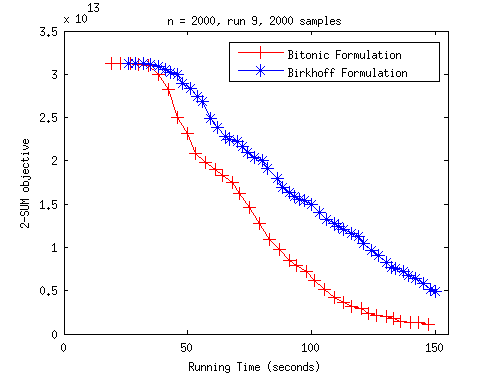}}
      \parbox[c]{55mm}{\includegraphics[width=54mm]{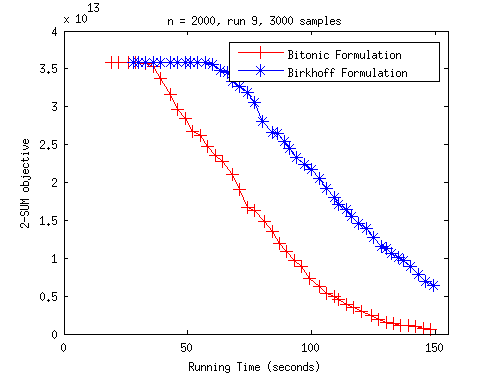}}\\ 
      
      \parbox[c]{52mm}{\includegraphics[width=54mm]{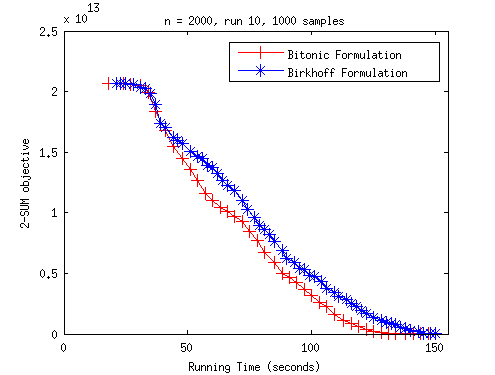}}
      \parbox[c]{52mm}{\includegraphics[width=54mm]{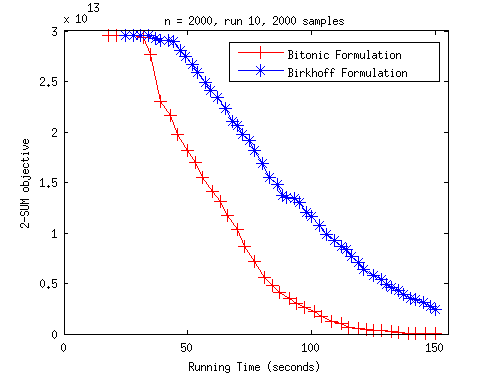}}
      \parbox[c]{55mm}{\includegraphics[width=54mm]{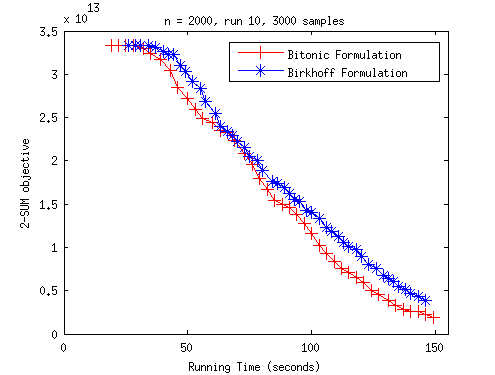}}\\   
    \end{tabular}
  \end{center}
  \caption{Linear Markov Chain --- Plot of the difference of the 2-SUM
    objective from the baseline objective over time for $n = 2000$ for
    runs 6 to 10.}
\end{figure}

\begin{figure}[hbtp]
  \begin{center}
    \begin{tabular}{ccc}
      \parbox[c]{52mm}{\includegraphics[width=54mm]{./fig/mc5000_grb/mc5000_grb_run_1_2500_samples_.png}}
      \parbox[c]{52mm}{\includegraphics[width=54mm]{./fig/mc5000_grb/mc5000_grb_run_1_5000_samples_.png}}
      \parbox[c]{55mm}{\includegraphics[width=54mm]{./fig/mc5000_grb/mc5000_grb_run_1_7500_samples_.png}}\\

      \parbox[c]{52mm}{\includegraphics[width=54mm]{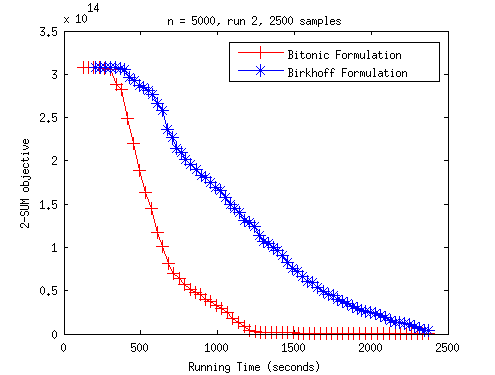}}
      \parbox[c]{52mm}{\includegraphics[width=54mm]{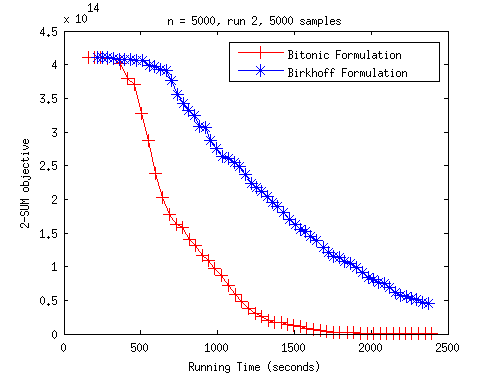}}
      \parbox[c]{55mm}{\includegraphics[width=54mm]{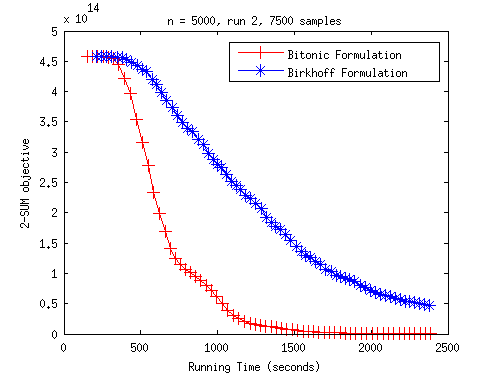}}\\

      \parbox[c]{52mm}{\includegraphics[width=54mm]{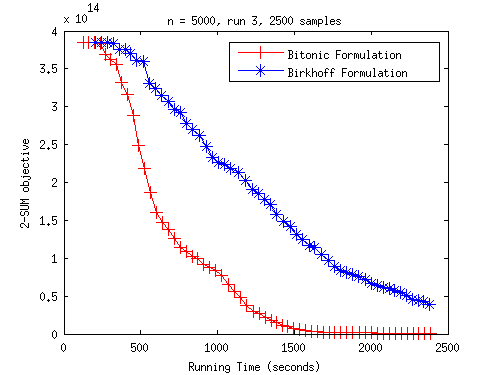}}
      \parbox[c]{52mm}{\includegraphics[width=54mm]{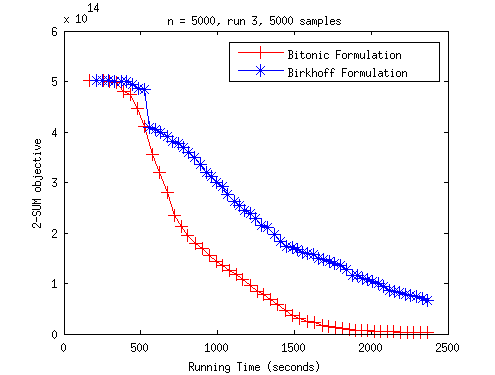}}
      \parbox[c]{55mm}{\includegraphics[width=54mm]{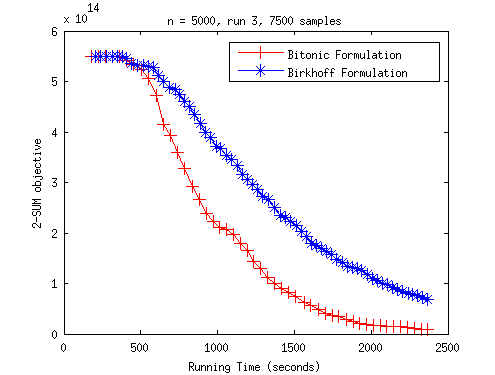}}\\

      \parbox[c]{52mm}{\includegraphics[width=54mm]{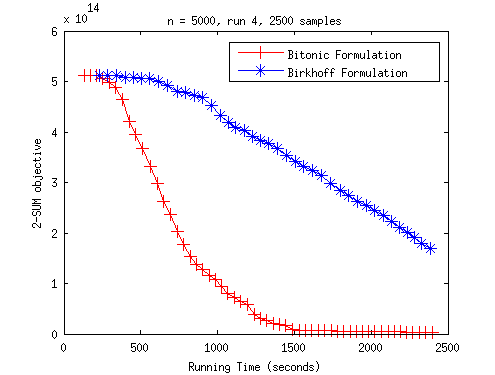}}
      \parbox[c]{52mm}{\includegraphics[width=54mm]{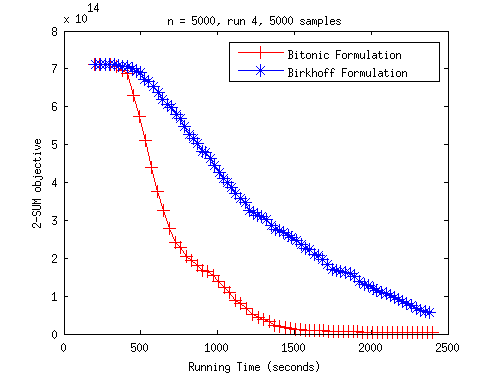}}
      \parbox[c]{55mm}{\includegraphics[width=54mm]{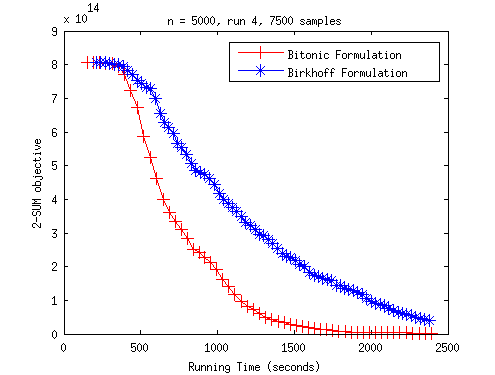}}\\ 
      
      \parbox[c]{52mm}{\includegraphics[width=54mm]{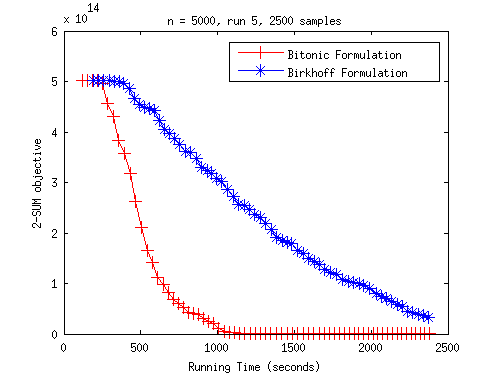}}
      \parbox[c]{52mm}{\includegraphics[width=54mm]{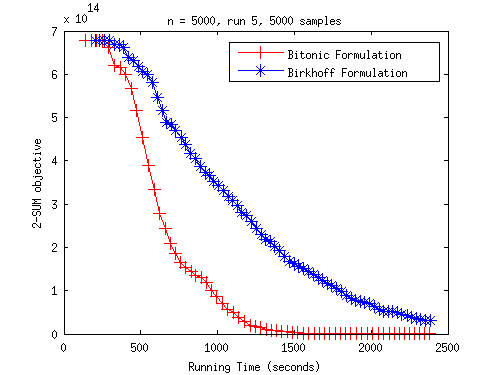}}
      \parbox[c]{55mm}{\includegraphics[width=54mm]{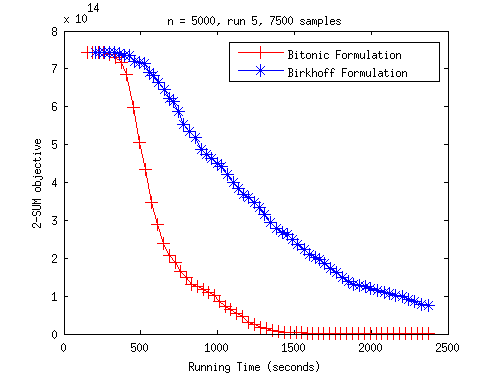}}
    \end{tabular}
  \end{center}
  \caption{Linear Markov Chain --- Plot of the difference of the 2-SUM
    objective from the baseline objective over time for $n = 5000$ for
    runs 1 to 5.}
\end{figure}

\begin{figure}[hbtp]
  \begin{center}
    \begin{tabular}{ccc}
      \parbox[c]{52mm}{\includegraphics[width=54mm]{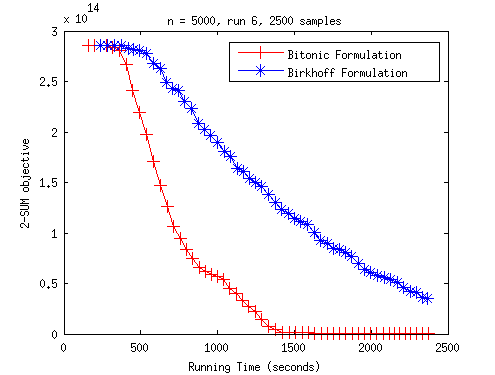}}
      \parbox[c]{52mm}{\includegraphics[width=54mm]{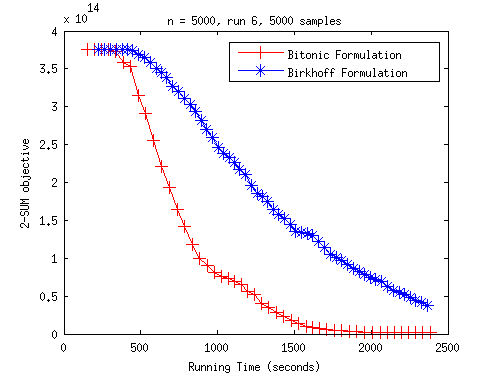}}
      \parbox[c]{55mm}{\includegraphics[width=54mm]{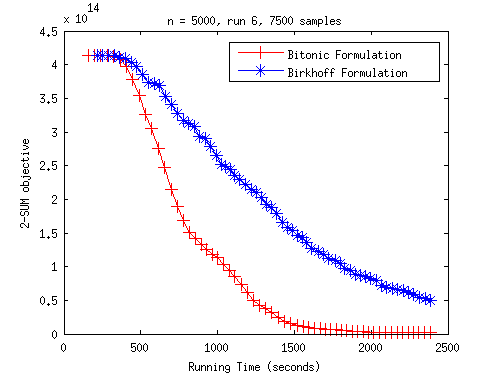}}\\ 
      
      \parbox[c]{52mm}{\includegraphics[width=54mm]{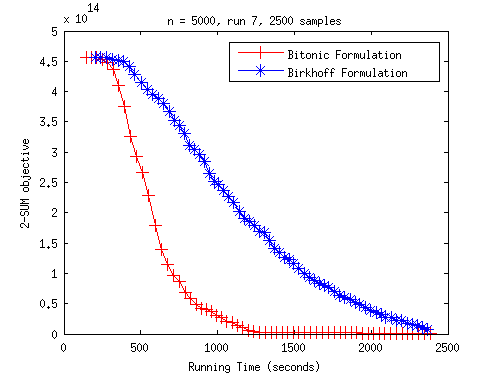}}
      \parbox[c]{52mm}{\includegraphics[width=54mm]{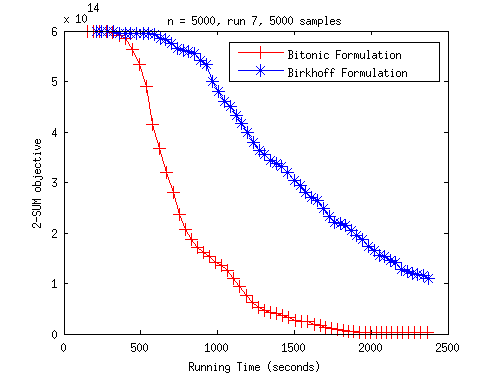}}
      \parbox[c]{55mm}{\includegraphics[width=54mm]{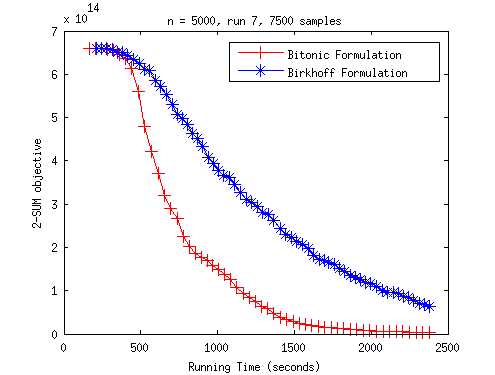}}\\ 
      
      \parbox[c]{52mm}{\includegraphics[width=54mm]{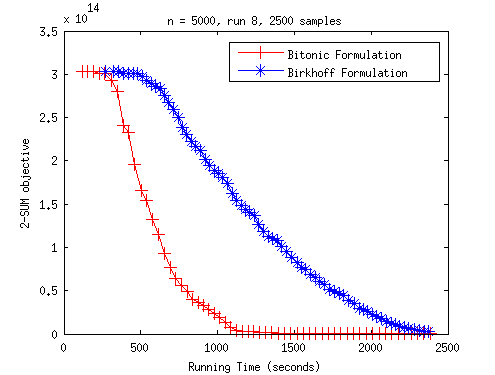}}
      \parbox[c]{52mm}{\includegraphics[width=54mm]{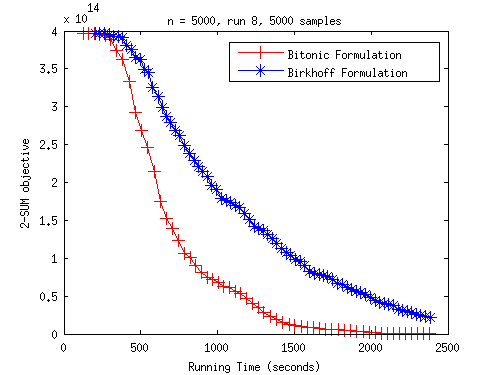}}
      \parbox[c]{55mm}{\includegraphics[width=54mm]{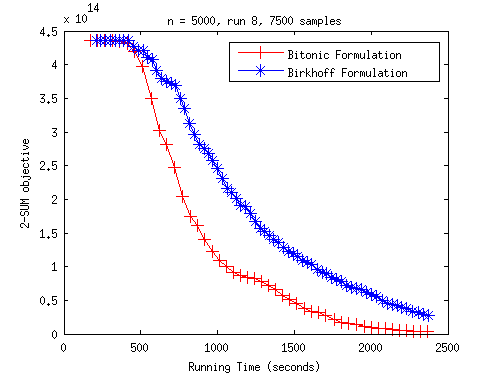}}\\ 
      
      \parbox[c]{52mm}{\includegraphics[width=54mm]{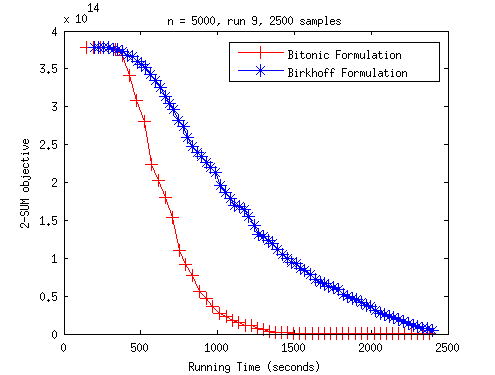}}
      \parbox[c]{52mm}{\includegraphics[width=54mm]{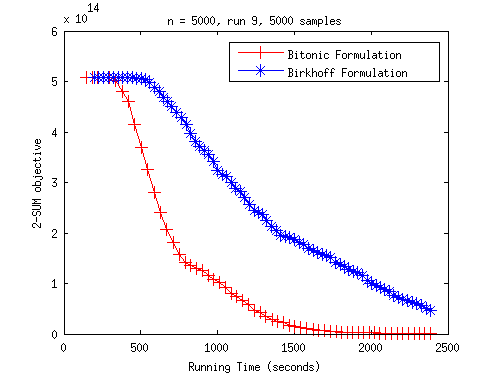}}
      \parbox[c]{55mm}{\includegraphics[width=54mm]{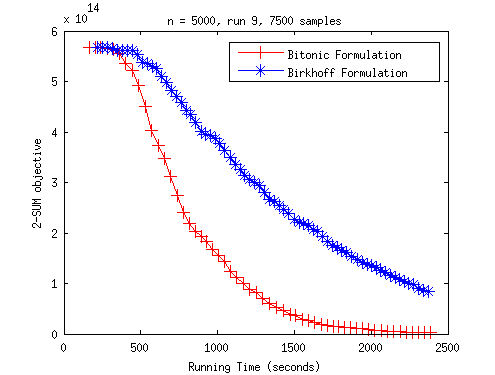}}\\
      
      \parbox[c]{52mm}{\includegraphics[width=54mm]{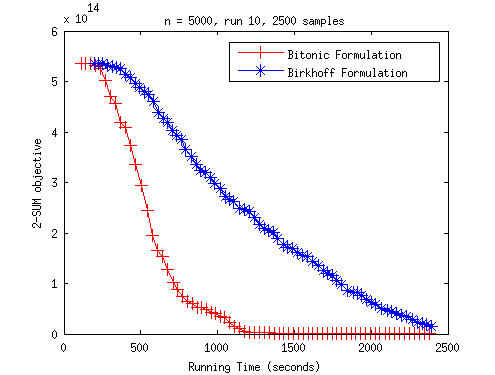}}
      \parbox[c]{52mm}{\includegraphics[width=54mm]{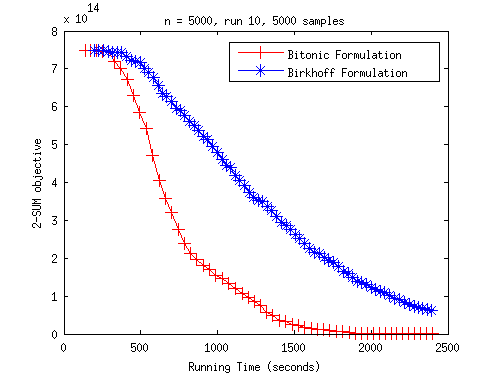}}
      \parbox[c]{55mm}{\includegraphics[width=54mm]{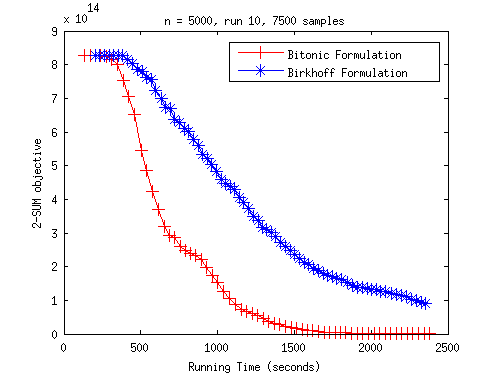}}
    \end{tabular}
  \end{center}
  \caption{Linear Markov Chain --- Plot of the difference of the 2-SUM
    objective from the baseline objective over time for $n = 5000$ for
    runs 6 to 10.}
\end{figure}

\end{document}